\documentclass[a4paper,10pt]{article}

\def\expensiveFigures{1}
\def\longtitle{Efficient presolving methods for the influence maximization problem}

\newcommand{\cmipoptauthors}{%
	Wei-Kun Chen
}

\makeatletter
\DeclareRobustCommand*{\escapeus}[1]{%
	\begingroup\@activeus\scantokens{#1\endinput}\endgroup}
\begingroup\lccode`\~=`\_\relax
\lowercase{\endgroup\def\@activeus{\catcode`\_=\active \let~\_}}
\makeatother
\usepackage{setspace}
\usepackage[figuresright]{rotating}
\usepackage[utf8]{inputenc}
\usepackage[font=small, labelfont=bf, margin=1cm]{caption}
\usepackage{amsmath,amsfonts,amssymb}
\usepackage{tabularx,booktabs}
\usepackage{url}
\usepackage{makecell}
\usepackage{setspace}
\usepackage{amsmath}
\usepackage{graphicx}
\usepackage{epstopdf}
\usepackage[textsize=small,textwidth=3.3cm]{todonotes}

\def\keywordname{{\bfseries Keywords}}%
\def\keywords#1{\par\addvspace\medskipamount{\rightskip=0pt plus1cm
		\def\and{\ifhmode\unskip\nobreak\fi\ $\cdot$
		}\noindent\keywordname\enspace\ignorespaces#1\par}}

\def\MSCname{{\bfseries Mathematics Subject Classification}}%
\def\MSC#1{\par\addvspace\medskipamount{\rightskip=0pt plus1cm
		\def\and{\ifhmode\unskip\nobreak\fi\ $\cdot$
		}\noindent\MSCname\enspace\ignorespaces#1\par}}

\usepackage{xspace}
\usepackage{etex}
\usepackage{enumitem}
\usepackage{dsfont}
\usepackage{mathtools}
\usepackage[boxed,ruled,linesnumbered]{algorithm2e}

\SetCommentSty{AlgoCommentStyle}
\usepackage{changes}
\usepackage{longtable}
\usepackage{tabularx}

\usepackage{geometry}
\usepackage{amsthm}
\usepackage{multirow}

\usepackage[breaklinks,colorlinks,citecolor=blue,linkcolor=blue,hyperfootnotes=false,
pdftitle={\longtitle},
pdfauthor={\cmipoptauthors},
pdfsubject={\longtitle}]{hyperref}

\usepackage[capitalize]{cleveref}
\usepackage[numbers]{natbib}

\usepackage{tikz}
\usepackage{tikz-3dplot}
\usepackage{tikzscale}
\usetikzlibrary{arrows}
\usetikzlibrary{shapes}
\usetikzlibrary{snakes}
\usetikzlibrary{patterns}
\usetikzlibrary{backgrounds,topaths}
\usetikzlibrary{arrows,decorations.markings,automata}

\usepackage{pgfplots}
\usepackage{subcaption}
\usepackage{textcomp}
\usepgfplotslibrary{groupplots}
\usepgfplotslibrary{units}

\usepackage{graphicx}
\usepackage{microtype}
\newcommand{\maxlength}{\mathbf{MaxReacSize}}

\newcommand{\st}{\,:\,}

\newcommand{\G}{\mathcal{G}}
\newcommand{\V}{\mathcal{V}}

\newcommand{\A}{\mathcal{A}}
\newcommand{\R}{\mathcal{R}}

\usepackage{siunitx}
\usepackage{xparse}
\ExplSyntaxOn
\def\myround#1{\num{\fp_eval:n {round(#1, 2)}}}
\ExplSyntaxOff

\definecolor{c1}{HTML}{000060}
\definecolor{c2}{HTML}{0000FF}
\definecolor{c3}{HTML}{36648B}
\definecolor{c4}{HTML}{4682B4}
\definecolor{c5}{HTML}{5CACEE}
\definecolor{c6}{HTML}{00FFFF}
\definecolor{c7}{HTML}{008888}
\definecolor{c8}{HTML}{00DD99}
\definecolor{c9}{HTML}{527B10}
\definecolor{c10}{HTML}{7BC618}
\definecolor{c11}{HTML}{33AA00}

\definecolor{scipoldcol}{HTML}{36648B}
\definecolor{scipnewcol}{HTML}{7BC618}

\definecolor{darkgreen}{HTML}{008800}

\theoremstyle{plain}

\setlist[itemize]{leftmargin=3.45ex}
\setlist[itemize,1]{label=$\bullet$,itemsep=0ex,topsep=0.9ex}
\setlist[itemize,2]{label=$\cdot$,topsep=0.5ex,leftmargin=2.75ex}
\setlist[enumerate]{leftmargin=3ex,itemsep=0.1ex,parsep=1ex,topsep=0.9ex}

\definecolor{tabcolor}{HTML}{6666AA}
\definecolor{f1}{HTML}{000060}
\definecolor{f2}{HTML}{0000FF}
\definecolor{f3}{HTML}{36648B}
\definecolor{f4}{HTML}{4682B4}
\definecolor{f5}{HTML}{5CACEE}
\definecolor{f6}{HTML}{00FFFF}
\definecolor{f7}{HTML}{00DD99}
\definecolor{f8}{HTML}{008888}
\definecolor{f9}{HTML}{000000}

\usepackage{float}
\newfloat{program}{thp}{lop}
\floatname{program}{Program}
\crefname{program}{program}{programs}

\newcommand{\inputExpensiveFigure}[1]{
	\ifthenelse{\expensiveFigures = 1}{\input{#1}}{}
}

\usepackage{titling}
\pretitle{\begin{center}\linespread{1.05}\Large}
	\posttitle{\par\end{center}\vskip 0.5em}
\preauthor{\begin{center}\linespread{1.15}\large \lineskip 0.6em\begin{tabular}[t]{c}}
		\postauthor{\end{tabular}\par\end{center}\vskip 0.6em}
\predate{\begin{center}}
	\postdate{\par\end{center}}
\thanksmarkseries{arabic}

\usepackage[]{titlesec}
\usepackage{etoolbox}
\makeatletter
\patchcmd{\ttlh@hang}{\parindent\z@}{\parindent\z@\leavevmode}{}{}
\patchcmd{\ttlh@hang}{\noindent}{}{}{}
\makeatother
\titleformat{\section}
{\normalfont\large\bfseries}{\thesection}{0.9ex}{}
\titleformat{\subsection}
{\normalfont\normalsize\bfseries}{\thesubsection}{0.9ex}{}
\titleformat{\subsubsection}
{\normalfont\normalsize\upshape}{\thesubsubsection}{0.9ex}{}
\titleformat{\paragraph}[runin]
{\normalfont\normalsize\itshape}{\theparagraph}{1em}{}
\titleformat{\subparagraph}[runin]
{\normalfont\normalsize\itshape}{\theparagraph}{1em}{}
\titlespacing*{\section}     {0pt}{21dd plus 8pt minus 4pt}{10.5dd}
\titlespacing*{\subsection}   {0pt}{21dd plus 8pt minus 4pt}{10.5dd}
\titlespacing*{\subsubsection}{0pt}{19dd plus 8pt minus 4pt}{10.5dd}
\titlespacing*{\paragraph}   {0pt}{13pt plus 8pt minus 4pt}{1em}
\titlespacing*{\subparagraph}   {0pt}{13pt plus 8pt minus 4pt}{1em}

\usepackage{supertabular}

\newcommand{\bz}{{\boldsymbol{z}}}

\newtheorem{theorem}{Theorem}[section]

\newtheorem{remark}[theorem]{Remark}
\newtheorem{observation}[theorem]{Observation}
\newtheorem{proposition}[theorem]{Proposition}

\newtheorem{lemma}[theorem]{Lemma}

\newcommand{\rev}[1]{{\color{black}#1}}

\newcommand{\revv}[1]{{#1}}

\newcommand{\Default}{{\texttt{Default}}\xspace}
\newcommand{\SCNA}{{\texttt{SCNA}}\xspace}
\newcommand{\INA}{{\texttt{SCNA+INA}}}

\newcommand{\sharpS}{{\texttt{\#S}}\xspace}
\newcommand{\Time}{{\texttt{T}}\xspace}
\newcommand{\Gap}{{\texttt{Gap}}\xspace}
\newcommand{\sharpN}{{\texttt{\#}{\texttt{N}}}\xspace}
\newcommand{\sharpC}{{\texttt{\#}{\texttt{C}}}\xspace}
\newcommand{\PTime}{{\texttt{PT}}\xspace}
\newcommand{\SepaTime}{{\texttt{ST}}\xspace}
\newcommand{\DeltaZ}{{\Delta\texttt{Z}}}
\newcommand{\DeltaR}{{\Delta\texttt{R}}}
\newcommand{\DeltaV}{{\Delta\texttt{V}}}
\newcommand{\DeltaA}{{\Delta\texttt{A}}}
\newcommand{\DeltaZINA}{{\Delta\texttt{Z}_\texttt{INA}}}
\newcommand{\DeltaRINA}{{\Delta\texttt{R}_\texttt{INA}}}
\newcommand{\DeltaZSCNA}{{\Delta\texttt{Z}_\texttt{SCNA}}}
\newcommand{\DeltaRSCNA}{{\Delta\texttt{R}_\texttt{SCNA}}}
\newcommand{\Mem}{{\texttt{M}}}
\newcommand{\TINA}{{\texttt{T}_\texttt{INA}}}

\begin{document}
	
	
	\title{\longtitle}
	
	\author{
		Sheng-Jie Chen\thanks{\itshape Academy of Mathematics and Systems Science, Chinese Academy of Sciences, Beijing 100190, China; School of Mathematical Sciences, University of Chinese Academy of Sciences, Beijing 100049, China, \nolinkurl{{shengjie_chen,dyh}@lsec.cc.ac.cn}\vspace*{0.5ex}}~,\hspace{3pt}
		Wei-Kun Chen\thanks{\itshape School of Mathematics and Statistics/Beijing Key Laboratory on MCAACI, Beijing Institute of Technology, Beijing 100081, China, \nolinkurl{chenweikun@bit.edu.cn}\vspace*{0.5ex}}
		\href{https://orcid.org/0000-0003-4147-1346}{\hspace{3pt}\includegraphics[width=10pt]{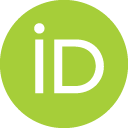}\hspace{3pt}},\\
		Yu-Hong Dai\thanksmark{1}	\href{https://orcid.org/0000-0002-6932-9512}{\hspace{3pt}\includegraphics[width=10pt]{ORCIDiD_icon128x128.png}},
		Jian-Hua Yuan\thanks{\itshape School of Science, Beijing University of Posts and Telecommunications, Beijing 100876, China, \nolinkurl{{jianhuayuan,houshan_zhang}@bupt.edu.cn}}\,,~
		Hou-Shan Zhang\thanksmark{3}
	}
	
	\date{\today}

	\newgeometry{left=38mm,right=38mm,top=35mm}

	\maketitle
	
	\begin{abstract}
		We consider the influence maximization problem (IMP) 
		which asks for identifying a limited number of key individuals to spread influence in a \revv{network} such that the expected number of influenced individuals is maximized.
		The stochastic maximal covering location problem (SMCLP) formulation is a mixed integer programming formulation that effectively approximates the IMP by the Monte-Carlo sampling.
		For IMPs with a large-scale network or a large number of samplings, however, the SMCLP formulation cannot be efficiently solved by existing exact algorithms due to its large problem size.  
		In this paper, we \revv{attempt to develop} presolving methods to reduce the problem size and hence enhance the capability of employing exact algorithms in solving large-scale IMPs.
		In particular, we propose two effective presolving methods, called strongly connected nodes aggregation (SCNA) and isomorphic nodes aggregation (INA), respectively.
		The SCNA enables to build a new SMCLP formulation that is potentially much more compact than the existing one, and the INA further eliminates variables and constraints in the SMCLP formulation.
		A theoretical analysis on two special cases of the IMP is provided to demonstrate the strength of the SCNA and INA in reducing the problem size of the SMCLP formulation.
		\revv{We} integrate the proposed presolving methods, SCNA and INA, into the Benders decomposition algorithm, which is recognized as one of the state-of-the-art exact algorithms for solving the IMP.
		\revv{We show that the proposed SCNA and INA provide the possibility to develop a much faster separation algorithm for the Benders cuts}. 
		Numerical results demonstrate that with the SCNA and INA, the Benders decomposition algorithm is much more effective in solving the IMP in terms of solution time.
		
		\MSC{90C10  \and 90C15}
		
	\end{abstract}

	\section{Introduction}
	\label{sect:introduction}
	Nowadays, with the popularity of online social network sites such as Facebook, Instagram, and Twitter, 
	propagation of influence in social networks has received more and more attention.
	Promotion of products, ideas, and specific behavior patterns can all be viewed as propagation of influence.
	In practice, influence spreads among individuals through the so-called ``word-of-mouth'' exchanges.
	{Individuals with more social connections can be seen as more influential, 
		which means they are more likely to exert influence on others.}
	In this setting, one related optimization problem, called the \emph{influence maximization problem} (IMP), 
	is to select a limited number of key individuals as a seed set, denoted as $\mathcal{S}$, to trigger a spread process in the social network 
	such that the expected number of influenced individuals is maximized after the spread.
	Mathematically, the IMP can be written as
	\begin{equation}\label{originalIMP}
		\max_{\mathcal{S}\subseteq\mathcal{V},~|\mathcal{S}|\leq K} \sigma(\mathcal{S}),
	\end{equation}
	where $\mathcal{V}$ is the set of individuals in the social network, $\mathcal{S}\subseteq\mathcal{V}$ ($|\mathcal{S}|\leq K\in \rev{\mathbb{Z}_{++}}$) is the seed set of key individuals that need to be identified, and $\sigma(\mathcal{S})$ is the influence function 
	measuring the expected number of individuals in the social network that can be influenced by the individuals in the seed set $\mathcal{S}$.

	\revv{The IMP plays a crucial role in various social network applications, such as viral marketing \cite{Chen2010,domingos2001mining} and rumor control \cite{budak2011limiting,he2012influence}. 
		For a comprehensive overview of this topic, we refer to \cite{R1Chen2013}.  
		Aside from the applications in social networks, similar concepts have been  investigated in other areas, including the spread of epidemics \cite{R3borrero2021scalable,R2dreyer2009irreversible}, network monitoring \cite{Leskovec2007}, habitat conservation \cite{R4sheldon2012maximizing}, and analysis of protein-interaction networks \cite{R5jo2016influence}.}
	Online \revv{network} tools enable \revv{collection of} a huge number of individuals and a huge amount of
	information about the network structures, and hence provide good opportunities to address these problems. 
	However, they also lead to large-scale \revv{networks} (with
	millions/billions of nodes and arcs), presenting new challenges \revv{in solving} large-scale IMPs.
	Therefore, \revv{development of} efficient algorithms to obtain a high-quality solution for large-scale
	IMPs is greatly needed.
	
	%
	%
	\subsection{Literature review}
	Kempe et al. \cite{Kempe2003} first proposed the discrete optimization problem formulation \eqref{originalIMP} for the IMP.
	Depending on different influence diffusion processes, they introduced two fundamental influence propagation models for the IMP: the {\emph{independent cascade model}} (ICM) and the {\emph{linear threshold model}} (LTM). 
	They showed that under both the ICM and LTM, the IMP is NP-hard, indicating that achieving an optimal solution of the IMP is challenging for the large-scale cases.
	By proving that the influence function $\sigma(\mathcal{S})$ is monotone and submodular, they were able to design a greedy algorithm, which starts with $\mathcal{S}=\varnothing$ and iteratively adds the individual with maximal marginal gain, 
	with an approximation ratio of $(1-1/e)$ (here $e$ denotes the base of the natural logarithm).
	Unfortunately, as shown in \cite{Chen2010,Chen2010scalable}, given a fixed seed set $\mathcal{S}$, it is \texttt{\#}P-hard to compute $\sigma(\mathcal{S})$ exactly.
	\rev{Therefore, Kempe et al. \cite{Kempe2003} proposed Monte-Carlo sampling, which provides a subset of equiprobable scenarios, of a reasonable size, to estimate $\sigma(\mathcal{S})$ (each scenario is represented by a \emph{live-arc graph}).}
	Following \cite{Kempe2003}, many researchers focused on \revv{the improvement of} the greedy algorithm and \revv{the development of} other heuristic algorithms for solving the IMP. 
	Specifically, Leskovec et al. \cite{Leskovec2007} utilized the submodularity of $\sigma(\mathcal{S})$ and presented an improved greedy algorithm called  \emph{cost-effective lazy forward}.
	According to their numerical results, their method is almost 700 times faster than the basic greedy algorithm of Kempe et al. \cite{Kempe2003}. 
	Chen et al. \cite{HEP} proposed another algorithmic enhancement for the greedy algorithm which reduces the graph searching time on computing \revv{the marginal gain (achieved by adding an individual into a seed set)}.
	Furthermore, they developed a much more efficient algorithm, called  \emph{degree discount}, for the IMP under the ICM that nearly matches the performance of the greedy algorithm.
	The \emph{two-phase influence maximization} \cite{Tang2014} and the \emph{influence maximization via martingales} \cite{Tang2015} heuristic algorithms also deserve special attention.
	They can not only guarantee an approximation ratio of  $(1-1/e)$, but also enable to solve large-scale IMPs in nearly linear time.
	We refer to \cite{Chen2010,Cheng2013,Galhotra2016,Kimura2006} for more greedy or heuristic algorithms for solving the IMP and \cite{Li2018} for a detailed comparison among different heuristic algorithms.
	
	Most heuristic algorithms find a suboptimal solution for the IMP with some worst case  guarantees.
	However, in some applications, it is crucial to identify an optimal solution instead of just a suboptimal one; see \cite{Guney2019}.
	As a result, using exact algorithms to solve the IMP has attracted more and more attention recently.
	{In most exact algorithms, a \emph{mixed integer programming} (MIP) formulation is established.}
	In particular, Wu and K{\"u}{\c{c}}{\"u}kyavuz \cite{Wu2017} transferred the IMP into the so-called \emph{two-stage stochastic submodular MIP model} and proposed a \emph{delayed constraint generation} algorithm to solve the problem to optimality.
	The computational results indicate that their algorithm is more efficient than the basic greedy algorithm in \cite{Kempe2003}, especially when $K$ is large.
	\rev{Given a collection of scenarios $\Omega$, G\"uney \cite{Guney2019}, G\"uney et al. \cite{Guney2020}, and Li et al. \cite{Li2019} formulated the IMP as a \emph{stochastic maximal covering location problem} (SMCLP) with $\mathcal{O}(|\mathcal{V}||\Omega|)$ variables and linear constraints.}
	%
	%
	Following this line, G\"{u}ney et al. \cite{Guney2020} developed a reformulation of the SMCLP and proposed a {\emph{Benders decomposition}} (BD) algorithm.
	Their experiment results show that the BD algorithm outperforms the one in \cite{Wu2017} by several orders of magnitude in terms of solution time.
	We refer to \cite{Fischetti2018,Gunnec2020,Kahr2020,kahr2022impact,Reghavan2019} for employing exact algorithms in solving several variants of the IMP.

	However, due to the NP-hardness of the IMP, the above exact algorithms are still inefficient, especially when the size of the \revv{network} or the number of scenarios is large. 
	\emph{Presolving} \cite{Gurobipresolve} is an appealing strategy to address this issue. 
	It removes redundant information and strengthens the model formulation with the aim of improving the performance of the subsequent solution procedure (e.g., the branch-and-cut or the BD approach).
	Indeed, presolving has been recognized as a standard routine of the state-of-the-art MIP solvers.
	%
	For problems with specific structures, developing customized presolving methods is often much more effective; see \cite{Borndorfer1998,Church2003,Heinz2013,Ljubic2012presolve} for using customized presolving methods to solve various problems. 
	In terms of the IMP, few articles are devoted to \revv{the development of} customized presolving methods.
	To the best of our knowledge, only two simple presolving methods have been developed in the literature \cite{Guney2020,Kahr2020} and they have been proved to be beneficial to solving the IMP in certain cases.
	\rev{Consequently, it is crucial to develop more customized presolving methods to further enhance the capability of using exact algorithms to solve large-scale IMPs.}

	\subsection{Contributions and outline}
	In this paper, we attempt to develop more presolving methods based on the SMCLP formulation to improve the solution efficiency for solving the IMP. 
	The main contributions of this paper are summarized as follows.
	\begin{itemize}
		\item 
		By exploiting the problem structure of the IMP, we propose two new presolving methods including (i) the \emph{strongly connected nodes aggregation} (SCNA) which aggregates nodes in each \emph{strongly connected component} (SCC), in a given live-arc graph, into a single virtual node; 
		and (ii) the \emph{isomorphic nodes aggregation} (INA), which extends the above idea to using \emph{isomorphic} nodes among different live-arc graphs for aggregation (two nodes in different live-arc graphs are called isomorphic if the nodes that can influence them are identical in the corresponding live-arc graphs).
		We show that the proposed presolving methods can effectively reduce the problem size of the SMCLP formulation. 
		In particular, after applying the SCNA, the IMP can be built on new live-arc graphs obtained by aggregating all SCCs in the original live-arc graphs, leading to a potentially much smaller SMCLP formulation.
		\item To demonstrate the strength of the proposed SCNA and INA in reducing the problem size of the SMCLP formulation, we provide a theoretical analysis on two special cases of the IMP: 
		one is built on the one-way bipartite \revv{network} under the LTM
		and the other one is built on the complete \revv{network} under the ICM.
		%
		For the first one, we give upper bounds, which are  linear with the size of the \revv{network} but independent of the number of \revv{scenarios}, for the numbers of variables and constraints in the reduced SMCLP formulation (obtained by applying the proposed presolving methods).
		For the second one, we provide a lower bound for the probability that after applying the SCNA and INA, there are only $|\mathcal{V}|+1$ variables and two linear constraints in the reduced SMCLP formulation.
		We show that such a probability can tend to one under certain conditions. 
		\item We integrate the SCNA and INA into the BD algorithm \cite{Guney2020}, which is recognized as one of the state-of-the-art exact algorithms to solve the IMP.
		We show that the proposed SCNA and INA provide the possibility to develop a much faster separation algorithm for the Benders cuts, as compared with the one in \cite{Guney2020}.
		\item Extensive numerical results on real-world \revv{networks} demonstrate that (i) the proposed SCNA and INA are quite effective in reducing the problem size of the SMCLP formulation; 
		(ii) 
		when integrating them into the BD algorithm, they can effectively speed up the solution procedure of the IMP.
	\end{itemize}

	The remainder of the paper is organized as follows. Section \ref{sect:pd} briefly reviews two fundamental influence propagation models (ICM and LTM) and the SMCLP formulation for the IMP.
	Section \ref{sect:presolving} presents the SCNA and INA and \revv{Section} \ref{sect:analysis} further shows their theoretical strength in reducing the problem size of the SMCLP formulation. 
	Section \ref{sect: al} describes the implementation of the \revv{INA} and the integration of \revv{the SCNA and INA} with the BD algorithm.
	Section \ref{sect: CR} provides the computational results.
	Section \ref{sect: EGP} studies a generalization of the IMP and shows that the proposed SCNA and INA can also be applied under some realistic conditions.
	Finally, \revv{Section} \ref{Sect:conclusions} gives some concluding remarks.

	\section{Propagation models and problem formulation}
	\label{sect:pd}
	In this section, we briefly review the propagation models and the SMCLP formulation for the IMP \cite{Guney2019,Guney2020,Li2019}.
	We use a directed graph $\mathcal{G}=(\mathcal{V},\mathcal{A})$ to refer to a \revv{network}, 
	in which a node $i \in \mathcal{V}$ represents the individual involved in the influence spread,
	and an arc $(i,j)\in\mathcal{A}$ represents that individual $i$ has the potential ability to influence (or activate) individual $j$. 
	The spread of influence in a given \revv{network} $\mathcal{G}$ needs to obey certain propagation rules.
	In \cite{Kempe2003}, Kempe et al. provided the following two fundamental influence propagation models called ICM and LTM. 
	\begin{itemize}
		\item 
		In the ICM, each arc $(i,j)\in\mathcal{A}$ is assigned an activation probability $\pi_{ij}$. The propagation process starts with a given seed set $\mathcal{S}$.
		If node $i$ has been activated at the beginning of step $t$, then during step $t$, it has a single chance to activate its (inactive) neighbor node $j$ with probability $\pi_{ij}$ independently. 
		\rev{If the activation is unsuccessful, node $i$ has no chance to influence node $j$ any more.}
		%
		Besides, for those nodes that are successfully activated during step $t$, they will remain active and attempt to activate their inactive neighbor nodes during step $t+1$. 
		When no more inactive nodes are activated, the diffusion process is terminated. %
		\item In the LTM, each arc $(i,j)\in\mathcal{A}$ is associated with a predefined weight $ b_{ij}\geq 0$ satisfying $\sum_{i\st(i,j)\in\mathcal{A}}b_{ij}\leq1$ for all $j \in \mathcal{V}$.   
		\revv{In addition}, each node \revv{$j\in\mathcal{V}$} selects a threshold value \revv{$\theta_j$}  randomly chosen from $[0,1]$ before the propagation process.
		Let $\mathcal{S}_t$ denote the activated nodes set at the beginning of step $t$ ($\mathcal{S}_0:=\mathcal{S}$ is the seed set). 
		Then an inactive node $j\in\mathcal{V}$ can be activated during step $t$ if and only if $\sum_{i\in \mathcal{S}_t}b_{ij}\geq\theta_j$, i.e., the total contribution of its active neighbors' influence weights exceeds its threshold value $\theta_j$. 
		In analogy to the ICM, if \revv{a node} is successfully activated during some step, it will remain active during the following steps, and the entire propagation process stops until no more nodes can be activated.
	\end{itemize}
	\rev{Given a seed set $\mathcal{S}$, the results (i.e., the distributions of influenced nodes) returned by the ICM and LTM can be different \cite{R1Chen2013}.
		For a comparison on the performance of the two models in different applications, we refer to \cite{akrouf2013social,li2017}.}
	
	Although the influence spread under the ICM or LTM is a stochastic process, Kempe et al. \cite{Kempe2003} showed that it can be equivalently converted into a (discrete) deterministic process. 
	More specifically, let $\Omega$ be the set of all possible scenarios of influence spread. Each scenario $\omega\in\Omega$ corresponds to a subgraph $\mathcal{G}^{\omega}=(\mathcal{V},\mathcal{A}^{\omega})$ of $\mathcal{G}$ (called a {live-arc} graph) with a probability $p^{\omega}$~(satisfying $\sum_{\omega\in\Omega}p^{\omega}=1$).
	Here arc $(i,j)\in\mathcal{A}^{\omega}$ indicates that in scenario $\omega$, if node $i$ is activated during the influence spread, then node $j$ must be activated by it.
	Let $\sigma^{\omega}(\mathcal{S})$ represent the number of  activated nodes in $\mathcal{G}^{\omega}$.
	Then, \revv{$\sigma^{\omega}(\mathcal{S})=|\left\{i\in\mathcal{V}\,:\, \text{there~exists~a~directed~path~in~graph~$\mathcal{G}^{\omega}$~from~a~node~in}~\mathcal{S}~\text{to~node}~i\right\}|$}. 
	%
	The influence function $\sigma(\mathcal{S})$ can equivalently be calculated by
	\begin{equation}\label{triggermodel}
		\sigma(\mathcal{S})=\sum_{\omega\in\Omega}p^{\omega}\sigma^{\omega}(\mathcal{S}).
	\end{equation}
	%
	We next discuss the computation of the probability $p^\omega$ of each scenario $\omega$ and the number of scenarios under the ICM or LTM.
	Under the ICM,  to construct a live-arc graph $\mathcal{G}^{\omega}$, each arc $(i,j) \in \mathcal{A}$ is independently determined to be live with probability $\pi_{ij}$.
	Hence, the probability of $\mathcal{G}^{\omega}$ is $p^{\omega}=\prod_{(i,j)\in\mathcal{A}^{\omega}}\pi_{ij}\prod_{(i,j)\in\mathcal{A}\backslash\mathcal{A}^{\omega}}(1-\pi_{ij})$ and the number of all possible live-arc graphs is $2^{|\mathcal{A}|}$.
	Under the LTM, for each node $j\in \mathcal{V}$, we select at most \revv{one incoming arc $(i,j)$ in $\G$} to be live with probability $b_{ij}$, and do not select any arc with the probability $1-\sum_{i\,:\,(i,j)\in\mathcal{A}}b_{ij}$.
	As a result, (i) each live-arc graph $\mathcal{G}^{\omega}$ has a probability $p^{\omega}=\prod_{j\in\mathcal{V}}I_j$, where $I_j:=b_{ij}$ if  $(i,j)\in\mathcal{A}^{\omega}$; 
	and $I_j :=1-\sum_{i\,:\,(i,j)\in\mathcal{A}}b_{ij}$, otherwise; and (ii) the number of all possible live-arc graphs is $\prod_{\revv{j} \in \mathcal{V}} \revv{(n_j+1)}$, where $\rev{n_j}$ denotes the number of incoming arcs of node \revv{$j$} in graph $\mathcal{G}$.
	It is worthwhile remarking that the stochastic IMP under the ICM or LTM  is equivalent to the IMP constructed via a finite number of live-arc graphs (in the sense that the distributions of nodes influenced by a given seed set are equivalent); see Kempe et al. \cite{Kempe2003}.

	Given a \revv{network} $\mathcal{G}$ and a set of scenarios $\Omega$, we next review the SMCLP formulation for the IMP \cite{Guney2019,Guney2020,Li2019}.
	First, for each scenario $\omega\in\Omega$ and node $i\in\mathcal{V}$, we denote $\mathcal{R}(\mathcal{G}^{\omega},i)$ as the reachability set of nodes that can activate node $i$ in live-arc graph $\mathcal{G}^{\omega}$ (i.e., \revv{$\mathcal{R}(\mathcal{G}^{\omega},i)=\{j\in\mathcal{V}\,:\,\text{there~exists~a~directed~path~in~graph~$\mathcal{G}^{\omega}$~from~node~}j~\text{to~node}~i\}$}).
	Then, for each $\omega\in\Omega$ and $i\in\mathcal{V}$, we introduce binary variables $y_i$ and $z_i^{\omega}$ to denote whether node $i$ is selected as a seed node and whether node $i$ can be activated in scenario $\omega$, respectively, i.e., 
	\begin{equation*}
		\begin{aligned}
			y_i =&\left\{\begin{array}{ll}1,
				& {\text{if~node}}~ i\in\mathcal{V} ~{\text{is chosen as a seed node}};~\\
				0,& {\text{otherwise;}}\end{array}\right.\\
			z_i^{\omega} =&\left\{\begin{array}{ll}1,
				& {\text{if~node}}~ i\in\mathcal{V} ~{\text{can be activated in scenario }}\omega\in\Omega;~\\
				0,& {\text{otherwise}}.\end{array}\right.
		\end{aligned}
	\end{equation*}
	Using the above notations, the authors in \cite{Guney2019,Guney2020,Li2019} formulated the IMP as the following SMCLP:
	\begin{subequations}\label{IMP}
		\begin{align}
			\max_{\boldsymbol{y},\,\boldsymbol{z}}\ & \sum_{\omega\in\Omega}p^{\omega}\sum_{i\in\mathcal{V}}z_i^{\omega} \label{maxspread}\\
			\hbox{s.t.}\ & \sum_{j\in \mathcal{R}(\mathcal{G}^{\omega},i)} y_j\geq z_i^{\omega},&&\forall ~\omega\in\Omega,~\forall ~i\in \mathcal{V}, \label{connectioncons} \\
			& \sum_{j\in \mathcal{V}}y_j \leq K, &&\label{budgetcons}\\
			& y_j \in\{0,1\} \label{ybincons}, &&  \forall ~j\in \mathcal{V},\\
			& z_i^{\omega} \in \{0,1\} \label{zbincons}, && \forall~ \omega\in\Omega,~\forall~ i\in \mathcal{V}.
		\end{align}
	\end{subequations}
	In this formulation, the objective function \eqref{maxspread} maximizes the expected number of influenced nodes in the \revv{network} $\mathcal{G}$.
	Reachability constraints \eqref{connectioncons} indicate that if node $i$ in scenario $\omega$  can be activated, then at least one of the nodes in its reachability set $\mathcal{R}(\mathcal{G}^{\omega},i)$ is chosen as a seed node. 
	Constraint \eqref{budgetcons} limits the cardinality of the set of seed nodes up to $K$. 
	Finally, constraints \eqref{ybincons} and \eqref{zbincons} restrict variables $\boldsymbol{y}$ and $\boldsymbol{z}$ to be binary.
	
	Unfortunately, formulation \eqref{IMP} is computationally intractable for the network with realistic dimensions due to the huge number of scenarios (for the IMP under the ICM and LTM, the numbers of all possible scenarios are both exponential).
	Hence, the Monte-Carlo sampling approach is often used to approximate the influence diffusion process in which a reasonable size of equiprobable scenarios set $\Omega'\subseteq \Omega$ is generated, 
	and the objective function in \eqref{maxspread} is replaced by $\sum_{\omega\in \Omega'}p'_\omega\sum_{i\in \mathcal{V}}z_i^\omega $ where $p'_\omega = 1/|\Omega'|$ and hence $\sum_{\omega\in \Omega'}p'_\omega=1$ \cite{Guney2019,Wu2017}.
	The rationale behind this is that from the approximation result in \cite{Kleywegt2002}, the probability of obtaining an optimal solution of the IMP \eqref{IMP} by solving the sampling version of the IMP converges to one exponentially fast as $|\Omega'| \rightarrow \infty$. 
	In practice, however, the selection of \revv{the number of scenarios} $|\Omega'|$ is crucial for the approximation quality of the sampling version of the IMP. 
	In general, the larger the $|\Omega'|$, the smaller the approximation error is. 
	We refer to \revv{Section 4.6 of \cite{Guney2019}  and Section 5.2 of \cite{Kahr2020}} for the empirical studies of the effect of \revv{the number of scenarios} on the approximation quality of the sampling version of the IMP \revv{and a variant of the IMP}, respectively. 
	Here we also want to highlight that the problem size of the sampling version of the IMP also grows linearly with the number of scenarios $|\Omega'|$ (as both numbers of variables and constraints are $\mathcal{O}(|\mathcal{V}||\Omega'|)$).
	This further makes it difficult to solve the problem by standard MIP solvers or the BD approach in \cite{Guney2020}, 
	especially when the size of graph $\mathcal{G}$ is also large.
	In the next section, we shall resolve this difficulty by proposing two new presolving methods to reduce the problem size of the SMCLP formulation \eqref{IMP}.
	
	In the remaining of this paper, we will consider the sampling version of the IMP.
	For simplicity of notations, we continue to use $\Omega$ and $p^{\omega}$ to represent the set of sampling scenarios and the probability of occurrence of \rev{sampling} scenario $\omega$, respectively.
	
	%
	
	
	\section{Two presolving methods}\label{sect:presolving}
	In this section, by exploiting the problem structure of formulation \eqref{IMP}, we propose two presolving methods to reduce the problem size of formulation \eqref{IMP}.
	Specifically, \revv{Section} \ref{sect: SCNA} studies the SCNA which aggregates the nodes in each SCC, in a given live-arc graph, into a single node, and \revv{Section} \ref{sect: INA} investigates the INA which extends the idea of the SCNA to applying isomorphic nodes aggregations among different live-arc graphs.

	\subsection{Strongly connected nodes aggregation}\label{sect: SCNA}
	\begin{figure}[H]
		\centering
		\includegraphics[width=2.1in,height=1.3in]{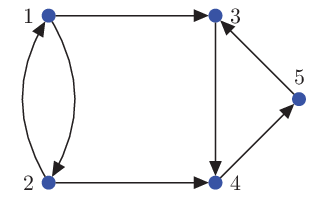}
		\caption{The example live-arc graph which includes two SCCs \{1,2\} and \{3,4,5\}.}
		\label{fig:SCNATikz1}
	\end{figure}

	In this subsection, we present a presolving method by considering the SCCs in a given live-arc graph.
	To begin with, we consider the example live-arc graph (corresponds to some scenario $\omega\in\Omega$) in Figure \ref{fig:SCNATikz1}.
	In this graph, there exists a directed path (an arc) from node $1$ to node $2$, 
	and as a result, if node $1$ is activated by some seed node, node $2$ can also be activated.
	Conversely, the fact that there exists a directed path (an arc) from node $2$ to node $1$ implies that if node $2$ is activated by some seed node, node $1$ can also be activated.
	This means that either (i) nodes $1$ and $2$ are simultaneously activated  by some seed node; or (ii) neither of them can be activated.
	Consequently, $z_1^\omega = z_2^\omega$ must hold in formulation \eqref{IMP}.
	This reveals some redundancy in formulation \eqref{IMP} as it uses two variables $z_1^\omega$ and $z_2^\omega$ 
	and two constraints in \eqref{connectioncons} without considering $z_1^\omega = z_2^\omega$ .
	Indeed, to simplify the problem formulation, we can remove variable $z_1^\omega$  and its corresponding reachability constraint in \eqref{connectioncons} and add the objective coefficient of variable $z_1^\omega$ into that of variable $z_2^\omega$.
	
	In general, for any given two strongly connected nodes $i_1,~i_2\in\mathcal{V}$ in a live-arc graph  $\mathcal{G}^\omega$ (i.e., there exists a directed path from node $i_1$ to node $i_2$ in $\mathcal{G}^{\omega}$ and vice versa), 
	we can remove one of the two variables and the corresponding reachability constraint in \eqref{connectioncons} from formulation \eqref{IMP}.
	%
	%
	Notice that for a given SCC in a live-arc graph, as all of its nodes are strongly connected, we can recursively apply the above argument until there remains only a single variable and a single constraint in \eqref{connectioncons} associated with this SCC.
	This provides us with the following presolving method.
	\begin{itemize}
		\item []
		{\bf SCNA}. For each live-arc graph $\mathcal{G}^\omega$, let $\{\mathcal{SC}_u^\omega\}_{u=1}^{\revv{s_\omega}}$, $\revv{s_\omega} \in \mathbb{Z}_{++}$, be all its SCCs. 
		For each SCC $\mathcal{SC}_u^\omega$, variables $z_{\revv{i}}^{\omega}$, $\revv{i} \in \mathcal{SC}_u^\omega$, 
		are first substituted by a new variable $z_u^\omega$ with its objective coefficient being $p^\omega |\mathcal{SC}_u^\omega|$.
		Then, all but one of constraints in \eqref{connectioncons} associated with variable $z_u^\omega$ are removed from the SMCLP formulation \eqref{IMP}.
	\end{itemize}
	To implement the SCNA, we only need to identify all SCCs in all live-arc graphs $\mathcal{G}^{\omega}$, $\omega \in \Omega$.
	For each graph $\mathcal{G}^\omega$, this can be done in linear time  $\mathcal{O}(|\mathcal{V}|+|\mathcal{A}^{\omega}|)$ using, e.g., the Kosaraju-Sharir's algorithm \cite{Sharir}.
	Consequently, the overall complexity to implement the SCNA for formulation \eqref{IMP} is $\mathcal{O}(\sum_{\omega\in \Omega}(|\mathcal{V}|+|\mathcal{A}^\omega|))
	$.
	%
	
	
	\begin{figure}[H]
		\centering
		\includegraphics[width=2.2in,height=0.36in]{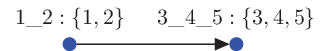}
		\caption{The compact live-arc graph obtained by aggregating each SCC into a single node of the graph in Figure \ref{fig:SCNATikz1}. Notice that there are only two nodes $1\_2$~(corresponds to SCC $\{1,2\}$) and $3\_4\_5$ (corresponds to SCC $\{3,4,5\}$) and one arc between these two nodes.}
		\label{fig:SCNATikz2}
	\end{figure}
	
	After applying the SCNA, the IMP can be equivalently constructed based on a new set of live-arc graphs, which are potentially much more compact than the original live-arc graphs.
	To be more specific, by aggregating each SCC of $\mathcal{G}^{\omega}$ into a single node with the weight being the size of the SCC, we can get a  directed acyclic graph, denoted as $\bar{\mathcal{G}}^{\omega}=(\bar{\mathcal{V}}^{\omega},\bar{\mathcal{A}}^{\omega})$.
	Each node $u$ in $\bar{\mathcal{V}}^{\omega}$ represents a distinct SCC in the original live-arc graph $\mathcal{G}^\omega$ and each arc $(u,v) \in \bar{\mathcal{A}}^{\omega}$ denotes that there exists an arc $(i,j)\in\A^{\omega}$ with $i\in\mathcal{SC}_u^{\omega}$ and $j\in\mathcal{SC}_v^{\omega}$ in the original live-arc graph $\mathcal{G}^\omega$
	(see Figure \ref{fig:SCNATikz2} for an example of this transformation of the graph in Figure \ref{fig:SCNATikz1}).
	As the reachability sets of the nodes inside a given SCC $\mathcal{SC}_u^{\omega}$ are identical, we use the notation $\mathcal{R}(\mathcal{G}^\omega,\mathcal{SC}_u^{\omega}) $ to represent the reachability set of this SCC, which is equal to each $\mathcal{R}(\mathcal{G}^\omega,\revv{i}) $, $\revv{i} \in \mathcal{SC}_u^{\omega} $.
	It follows immediately that 
	\begin{equation}
		\label{relation_reachset}
		\mathcal{R}(\mathcal{G}^\omega,\mathcal{SC}_u^{\omega}) = \bigcup_{v \in \mathcal{R}(\bar{\mathcal{G}}^\omega, u)}\mathcal{SC}_v^\omega,
	\end{equation} 
	where $\mathcal{R}(\bar{\mathcal{G}}^\omega, u)$ is the reachability set of node $u$ in live-arc graph $\bar{\mathcal{G}}^\omega$.
	Then constraints \eqref{connectioncons} reduce to 
	\begin{equation}
		\label{strongconnectioncons}
		\sum_{v\in\mathcal{R}(\bar{\mathcal{G}}^{\omega},u)}y(\mathcal{SC}_v^{\omega})\geq z_u^{\omega},~\forall~ \omega\in\Omega, ~\forall~ u\in\bar{\mathcal{V}}^{\omega}, 
	\end{equation}
	where $y(\mathcal{SC}_v^{\omega}) := \sum_{j \in \mathcal{SC}_v^{\omega}} y_j$.
	%
	Based on the above notations, the reduced SMCLP formulation after applying the SCNA can be written as
	\begin{equation}\label{strongIMP}
		\begin{aligned}
			\max_{\boldsymbol{y},\,\boldsymbol{z}}\ & \sum_{\omega\in\Omega}p^{\omega}\sum_{u\in\bar{\mathcal{V}}^{\omega}}|\mathcal{SC}_u^{\omega}|z_{u}^{\omega} \\
			\hbox{s.t.}\ & 
			\eqref{budgetcons},~\eqref{ybincons},~\eqref{strongconnectioncons},\\
			& z_{u}^{\omega} \in \{0,1\} , && \forall ~\omega\in\Omega, ~\forall~ u\in \bar{\mathcal{V}}^{\omega}.
		\end{aligned}
	\end{equation}
	It is worthwhile remarking that the numbers of variables $z_u^\omega$ and the corresponding reachability constraints in the reduced formulation are equal to $\sum_{\omega \in \Omega} |\mathcal{\bar{V}}^\omega|$ (which is the number of the SCCs in the original live-arc graphs).
	This can be potentially much smaller than those in formulation \eqref{IMP}, especially for the case where the numbers of SCCs are much smaller than the numbers of nodes in the  live-arc graphs. 
	As a result, it can be expected that solving formulation \eqref{strongIMP}  is much more efficient than solving formulation \eqref{IMP}.
	In addition, the fact that formulation \eqref{strongIMP} is built on the (potentially) compact and directed acyclic live-arc graphs \revv{plays an important role} in improving the performance of the BD algorithm (see \revv{Section \ref{sect: bdi}} further ahead).

	\subsection{Isomorphic nodes aggregation}\label{sect: INA}
	The SCNA performs reductions on two nodes $i_1$ and $i_2$ in a given live-arc graph $\mathcal{G}^\omega$ where nodes $i_1$ and $i_2$ are strongly connected, or equivalently, the reachability sets of nodes $i_1$ and $i_2$ are identical, i.e.,  $\mathcal{R}(\mathcal{G}^{\omega},i_1)=\mathcal{R}(\mathcal{G}^{\omega},i_2)$.
	In this subsection, we concentrate on \revv{the extension of} the result to the {isomorphic} nodes among different live-arc graphs. 
	As it has been previously mentioned, two nodes $i_1$ and $i_2$ in two different live-arc graphs $\mathcal{G}^\omega$ and $\mathcal{G}^\eta$ are called isomorphic if their reachability sets are identical, i.e.,
	\begin{equation}\label{SCequalcondi_equation}
		\mathcal{R}(\mathcal{G}^{\omega},i_1)=\mathcal{R}(\mathcal{G}^{\eta},i_2).
	\end{equation}
	
	We begin with the following observation stating that the values of variables $\boldsymbol{z}$ are determined by the values of variables $\boldsymbol{y}$ in formulation \eqref{IMP}.
	\begin{observation}\label{projection_on_z}
		There must exist an optimal solution $(\bar{\boldsymbol{y}},\bar{\boldsymbol{z}})$ of formulation \eqref{IMP} such that
		\begin{equation}
			\bar{z}_i^{\omega}=\min\left\{1,\sum\limits_{j\in\mathcal{R}(\mathcal{G}^{\omega},i)}\bar{y}_j\right\},~ \forall~ \omega \in \Omega,~\forall~ i\in\mathcal{V}.~\label{projection_on_z_equation}
		\end{equation}
	\end{observation}
	\noindent In particular, if node $i\in\mathcal{V}$ does not have any incoming arc in $\mathcal{G}^{\omega}$ for some $\omega\in\Omega$, 
	i.e., $\mathcal{R}(\G^\omega , i)=\{i\}$ is a singleton,
	the associated reachability constraint in \eqref{connectioncons} reduces to $y_i\geq z_i^{\omega}$. By Observation \ref{projection_on_z}, we can get $z_i^{\omega}=\min\{1,y_i\}=y_i$.
	As a result, we can perform a reduction on formulation \eqref{IMP} by aggregating $z_i^\omega:=y_i$ and removing the associated constraint in \eqref{connectioncons}.
	We call this reduction the \emph{singleton node aggregation} (SNA).
	Indeed, this is exactly the ``P2'' presolving method proposed in G\"uney et al. \cite{Guney2020}. 
	
	We next use Observation \ref{projection_on_z} to derive the INA.
	Let $i_1$ and $i_2$ be two isomorphic nodes in two different live-arc graphs $\mathcal{G}^\omega$ and $\mathcal{G}^\eta$.
	By Observation \ref{projection_on_z}, 
	there must exist an optimal solution $(\boldsymbol{\bar{y}},\boldsymbol{\bar{z}})$ such that
	$$\bar{z}_{i_1}^{\omega} = \min\left \{1, \sum_{j\in \mathcal{R}(\mathcal{G}^{\omega},i_1)}\bar{y}_j\right \}~~\text {and}~~\bar{z}_{i_2}^{\eta} = \min\left\{1, 
	\sum_{j\in \mathcal{R}(\mathcal{G}^{\eta},i_2)}\bar{y}_j\right\}.$$
	By \eqref{SCequalcondi_equation}, we have $\bar{z}_{i_1}^{\omega} = \bar{z}_{i_2}^{\eta}$.
	This implies that setting $z_{i_1}^{\omega}:=z_{i_2}^{\eta}$ in formulation \eqref{IMP} does not change its optimal value.
	Consequently, we have the following presolving method.
	\begin{itemize}
		\item [] {\bf INA}. If nodes $i_1$ and $i_2$  in two different live-arc graphs $\mathcal{G}^\omega$ and $\mathcal{G}^\eta$ are isomorphic,  variable $z_{i_1}^{\omega}$ can be replaced by variable $z_{i_2}^{\eta}$  and  constraint $
		\sum_{j\in \mathcal{R}(\mathcal{G}^{\omega},i_1)}y_j\geq z_{i_1}^{\omega}
		$  can be removed
		from formulation \eqref{IMP}.
	\end{itemize}
	The SCNA can be regarded as a special case of the INA, which is restricted to aggregating the isomorphic (strongly connected) nodes inside each live-arc graph.
	However, as it has been discussed in \revv{Section} \ref{sect: SCNA}, implementing the SCNA can be done in $\mathcal{O}(\sum_{\omega\in \Omega}(|\mathcal{V}|+|\mathcal{A}^\omega|))$, 
	which is much faster than that of implementing the INA.
	The latter requires to check whether or not condition \eqref{SCequalcondi_equation} holds for all 4-tuples $(\omega,\eta,i_1,i_2 )$ with an overall complexity of $\mathcal{O}({|\Omega|}^2{|\mathcal{V}|^3})$.
	This shows that, to implement the INA, it is better to first implement the SCNA, and then detect isomorphic nodes among different scenarios based on the compact formulation \eqref{strongIMP}  (in \revv{Section} \ref{subsect: im_pre}, we shall provide a fast heuristic algorithm for implementing the INA).
	After applying the INA on formulation \eqref{strongIMP}, the formulation of the IMP can be presented as follows:
	\begin{equation}\label{IMP_INA}
		\begin{aligned}
			\max_{\boldsymbol{y},\,\boldsymbol{z}}\ & \sum_{\omega\in\Omega}\sum_{u\in\tilde{\mathcal{V}}^{\omega}}f_u^{\omega}z_u^{\omega} \\
			\hbox{s.t.}\ & 	\eqref{budgetcons},~\eqref{ybincons},&&\\
			& \sum_{v\in\mathcal{R}(\bar{\mathcal{G}}^{\omega},u)}y(\mathcal{SC}_v^{\omega})\geq z_u^{\omega}, &&\forall~ \omega\in\Omega, ~\forall~ u\in\tilde{\mathcal{V}}^{\omega}, \\
			& z_{u}^{\omega} \in \{0,1\}, && \forall~ \omega\in\Omega, ~\forall~ u\in \tilde{\mathcal{V}}^{\omega},
		\end{aligned}
	\end{equation}
	where $\tilde{\mathcal{V}}^\omega$ denotes the set of nodes in $\bar{\mathcal{G}}^\omega$ that are not aggregated by \revv{some other nodes in other live-arc graphs} and $f_u^\omega$  denotes the objective coefficient of variable $z_{u}^\omega$ after applying the INA.
	For simplicity, for two isomorphic nodes $i_1$ and $i_2$ in two different live-arc graphs $\mathcal{G}^\omega$ and $\mathcal{G}^\eta$, we aggregated variable $z_{i_1}^\omega$ by  variable $z_{i_2}^\eta$ if $\omega > \eta$.
	It is worthwhile to highlight that, with the increasing number of scenarios $|\Omega|$, a node in a live-arc graph $\bar{\mathcal{G}}^\omega$ is more likely to be aggregated by other isomorphic nodes (in other live-arc graphs), and 
	consequently,  $|\tilde{\mathcal{V}}^\omega|$ in formulation \eqref{IMP_INA} tends to be smaller.  
	
	\begin{remark}
		The SNA can be used to further simplify formulation \eqref{IMP_INA}. In particular, if $\bigcup_{v\in\mathcal{R}(\bar{\mathcal{G}}^{\omega},u)}\mathcal{SC}^{\omega}_v=\{j\}$,  we can aggregate $z_u^\omega:=y_j$ and remove the corresponding reachability constraint in \eqref{IMP_INA}.
		%
	\end{remark}
	
	\section{Theoretical analysis}\label{sect:analysis}

	In this section, we demonstrate the strength of the proposed SCNA and INA in reducing the problem size of the SMCLP formulation \eqref{IMP} by analyzing two special cases of the IMP.
	In particular, for the first case where the IMP is built upon a one-way bipartite \revv{network} under the LTM, 
	we \revv{provide} upper bounds, 
	which are linear with the size of the \revv{network} but independent of the number of scenarios,
	for the numbers of variables and constraints in the reduced SMCLP formulation (obtained by applying the proposed presolving methods).
	\rev{For the second case where the IMP is built upon a complete  \revv{network} under the ICM, we provide a lower bound for the probability that after applying the SCNA and INA, there are only $|\mathcal{V}|+1$ variables and two linear constraints in the reduced SMCLP formulation \eqref{IMP_INA}. Such a probability can tend to one under certain conditions.}

	\subsection{One-way bipartite \revv{network} under the LTM}\label{subsect:obg}
	{
		A one-way bipartite \revv{network} is a bipartite graph in which all arcs are from one side (the source nodes) to another side (the target nodes). 
		One-way bipartite \revv{networks} also arise from several applications \cite{alon2012optimizing,berardo2014,ERGUN2002,hatano2016adaptive,soma2014optimal}.
		For example, the authors in \cite{alon2012optimizing,hatano2016adaptive,soma2014optimal} considered one of the major decisions in a marketing plan that deals with the allocation of a given budget among media channels (i.e., source nodes) in order to maximize the influence on a set of potential customers (i.e., target nodes), which can be characterized by a one-way bipartite \revv{network}.
		\rev{Other applications on the one-way bipartite \revv{networks}} include, e.g., the human sexual contact network in \cite{ERGUN2002} (where the nodes denote the groups of two different
		genders and the arcs denote the sexual connections between males and females) and the collaboration
		network in \cite{berardo2014} (the source and target nodes denote  the organizations and the projects, respectively, and an arc between organization $i$ and project $j$ denotes that organization $i$ participates in project $j$).
		\rev{Due to the simple structure, some theoretical properties on the influence propagation in one-way bipartite \revv{networks} have also been established.
			In particular, it is possible to compute the exact influence coverage $\sigma(S)$ by a dynamic programming procedure; see \cite{Wu2019,Zhang2014}.}
		We remark that most existing works assume that the influence spreads from the source nodes to the target nodes.
		In this subsection, we consider the generalized case where the source nodes and target nodes can also exert influence to themselves. 
		In the following, we consider the IMP built upon such a one-way bipartite \revv{network} under the LTM.
		
		Let $\mathcal{G}_{\rm{B}}=(\mathcal{M}\cup \mathcal{N},\mathcal{A})$ be a given one-way bipartite graph.
		All arcs in $\mathcal{A}$ are from the source nodes set $\mathcal{M}$ to the target nodes set $\mathcal{N}$.
		For each scenario $\omega\in\Omega$, we denote its live-arc graph as $\mathcal{G}_{\rm{B}}^{\omega}=(\mathcal{M}\cup \mathcal{N},\mathcal{A}^{\omega})$. 
		Notice that under the LTM, in each live-arc graph $\mathcal{G}_{\rm{B}}^\omega$, each node $i$, $i \in \mathcal{M}$, does not have any incoming arc and each node $j$, $j \in \mathcal{N}$, has at most one incoming arc.
		Therefore, under the LTM, the reachability set of a node $i\in \mathcal{M}$ is $\mathcal{R}(\mathcal{G}_{\rm{B}}^\omega, i) = \{i\}$, and the reachability set of a node $j\in \mathcal{N}$ is
		\begin{equation}\label{obg_property}
			\begin{aligned}
				\mathcal{R}(\mathcal{G}_{\rm{B}}^{\omega},j) =\left\{
				\begin{array}{ll}
					\{j\},& {\text{if node}}~j~\text{does not have any incoming arc in $\G_{\rm{B}}^{\omega}$};\\[3pt]
					\{i,j\},& {\text{if there exists some node}}~i\in \mathcal{M}~\text{such that}~(i,j)\in\mathcal{A}^{\omega}.\\
				\end{array}\right.
			\end{aligned}
		\end{equation}
		Next, for each node $i \in \mathcal{M} 
		\cup \mathcal{N}$, we define a set of scenarios
		\begin{equation}
			\begin{aligned}
				& \Omega(i):=
				\left\{\omega\in\Omega \,:\, \mathcal{R}(\mathcal{G}_{\rm{B}}^{\omega},i)=\{i\}\right\},\\
			\end{aligned}
		\end{equation}
		and for each arc $(i,j) \in \mathcal{A}$, we define another set of scenarios 
		\begin{equation}
			\Omega(i,j):=\left\{\omega\in\Omega\,:\,\mathcal{R}(\mathcal{G}_{\rm{B}}^{\omega},j)=\{i,j\}\right\}.
		\end{equation}
		By definition, it follows that
		\revv{
			\begin{equation}
				\label{omegarelation}
				\Omega=\left\{
				\begin{aligned}
					&\Omega (i),  && ~\text{for}~i \in \mathcal{M};\\
					&  \left (\bigcup_{i\in \mathcal{M}}\Omega (i,j)\right ) \cup \Omega (j), &&   ~\text{for}~j \in \mathcal{N}.
				\end{aligned}\right.
			\end{equation}
		}%
		Then, we have the followings:
		\begin{itemize}
			\item [(i)]
			by applying the SNA for each $i \in \mathcal{M}\cup \mathcal{N}$, we can aggregate  $z_i^\omega:=y_i $ for all $\omega \in \Omega(i)$ and remove the corresponding  constraints in \eqref{connectioncons};
			\item [(ii)]
			by applying the INA for each $(i,j) \in \mathcal{A}':=\left\{(i,j)\in\mathcal{A}\,:\,\Omega(i,j)\neq\varnothing\right\}$, we can aggregate all variables $z_j^{\omega}$, $\omega\in\Omega(i,j)$, into a single variable, denoted as $z_{ij}$, and remove the redundant constraints in \eqref{connectioncons}.
		\end{itemize}
		As a result, the reduced formulation is given by
		\begin{subequations}\label{IMP_obg_presolve}
			\begin{align}
				\max_{\boldsymbol{y},\,\boldsymbol{z}}\ & \sum_{i\in\mathcal{M}\cup \mathcal{N}}s_iy_i+\sum_{(i,j)\in\mathcal{A}'}c_{ij}z_{ij} \label{obg_obj}\\
				\hbox{s.t.}\ &y_i+y_j\geq z_{ij},&&\forall~ (i,j)\in\mathcal{A}', \label{connections_obg} \\
				& \sum_{i\in\mathcal{M}\cup \mathcal{N}}y_i  \leq K, \label{cardinality}\\
				& y_i \in\{0,1\}, && \forall~ i\in \mathcal{M}\cup \mathcal{N},\label{ybinary_obj}\\
				& z_{ij}\in \{0,1\} , && \forall ~(i,j)\in\mathcal{A}',\label{vbinary_obg}
			\end{align}
		\end{subequations}
		where $s_i := \sum_{\omega \in \Omega(i)} p^\omega$ for $i \in \mathcal{M}\cup \mathcal{N}$ 
		and $c_{ij} := \sum_{\omega \in \Omega(i,j)} p^\omega $ for $(i,j) \in \mathcal{A}'$, respectively.
		%
		%
		Since $|\mathcal{A}'| \leq |\mathcal{A}|$, we have the following theorem providing upper bounds for the numbers of variables and constraints in the reduced SMCLP formulation \eqref{IMP_obg_presolve}.
		
		
		\begin{theorem}
			\label{bipartitle}
			Consider the IMP on the one-way bipartite \revv{network} $\mathcal{G}_{\rm {B}}$ with a finite set of scenarios $\Omega$ under the LTM.
			Applying the SNA and INA on formulation \eqref{IMP}, the numbers of variables and constraints in the reduced SMCLP formulation \eqref{IMP_obg_presolve} are at most $|\mathcal{M}|+|\mathcal{N}|+|\mathcal{A}|$ and $|\mathcal{A}|+1$, respectively.
		\end{theorem}

		Finally, we provide more analysis results for formulation \eqref{IMP_obg_presolve}. To proceed, we note that using \eqref{omegarelation} and $\sum_{\omega\in\Omega}p^{\omega}=1$, we have the following properties on the objective coefficients of formulation \eqref{IMP_obg_presolve}.
		\begin{remark}
			\label{defsc}
			(i) $s_i = 1$ for all $i \in \mathcal{M}$; and (ii) $s_j + \sum_{i\,:\,(i,j)\in\mathcal{A}'}c_{ij}=1$ for all $j \in \mathcal{N}$.
		\end{remark}
		
		\begin{proposition}
			\label{theorem:obj}
			The linear programming (LP) relaxation of formulation \eqref{IMP_obg_presolve} is tight.
			Moreover, formulation \eqref{IMP_obg_presolve} can be solved in strongly \revv{polynomial time}.
		\end{proposition}
		\begin{proof}
			The proof \revv{is given in} Appendix \ref{proof_bipartite}.
		\end{proof}
		
		
		%

		\subsection{Complete \revv{network} under the ICM}\label{subsect:cg}
		
		In this subsection, we study another special case of the IMP where the considered \revv{network} is a complete graph (denoted as $\mathcal{G}_{\rm C}=(\mathcal{V}_{\rm C}, \mathcal{A}_{\rm C})$) and the influence propagation model is the ICM.
		{Let  $\G_{\rm{C}}^{\omega}=(\V_{\rm{C}},\A_{\rm{C}}^{\omega})$ be a live-arc graph and denote $n= |\V_{\rm{C}}|$.}
		Recall that for the live-arc graph constructed under the ICM, each arc is determined to be live independently with probability $\pi_{ij}$.
		Consequently, $\G_{\rm{C}}^{\omega}$  can be seen as a directed \emph{Erd\H{o}s-R\'{e}nyi} (ER) random graph \cite{cao2020connectivity,detering2019Boot}.
		If the arc probabilities are homogeneous, i.e., $\pi_{ij} = p$ for all $(i,j) \in \A_{\rm{C}}$ and some $p \in(0,1]$, $\G_{\rm{C}}^{\omega}$ is a homogeneous directed ER random graph; otherwise it is an inhomogeneous directed ER random graph.
		Homogeneous directed ER random graph is shown to be strongly connected with a probability \rev{tending to} one (as $n\rightarrow\infty$) under certain conditions; see, e.g., \cite{graham2008note}.
		{The following lemma further provides a lower bound for the probability of the strong connectivity of $\G_{\rm{C}}^{\omega}$ with respect to the number of nodes $n$ and the arc probability $p$.}
		
		\begin{lemma}\label{theorem: strongLB}
			Suppose that $\pi_{ij} = p \in (0,1]$ for all $(i,j) \in \A_{\rm{C}}$ and 
			\begin{equation}\label{assumptionSCC}
				\max\left\{(n-1)(1-p^2)^{\frac{n}{2}+1}, 2(1-p^2)^{\frac{3n}{16}-1}\right\} \leq 1.
			\end{equation}
			Then 
			\begin{equation}\label{ERlimproperty_equation}
				\mathbb{P}(\text{Graph}~\mathcal{G}_{\rm{C}}^\omega\text{ is strongly connected}) \geq 1 - n(n-1)(1-p^2)^{n-1}.
			\end{equation}
		\end{lemma}
		\begin{proof}
			The proof can be found in \revv{Section} 1 of \cite{Chen2022}.
		\end{proof}

		\begin{theorem}
			\label{ERCorollary}
			Consider the IMP on the complete \revv{network} $\mathcal{G}_{\rm {C}}$ with a finite set of scenarios $\Omega$ under the ICM.
			Suppose that $\pi_{ij}\geq  p~\revv{\in (0,1]}$ for all $(i,j) \in \mathcal{A}_{\rm{C}}$ and \eqref{assumptionSCC} holds.
			Then, by applying the SCNA and INA, there are only $n+1$ variables and two linear constraints in the reduced SMCLP formulation \eqref{IMP_INA} with a probability at least $p^*$ where
			\begin{equation}
				\label{defpstar}
				p^*=(1 - n(n-1)(1-p^2)^{n-1})^{|\Omega|}.
			\end{equation}
		\end{theorem}
		\begin{proof}
			Notice that the probability that $\mathcal{G}^\omega_{\rm{C}}$ is strongly connected with $\pi_{ij} \geq p$ for all $(i,j) \in \A_{\rm{C}}$ is larger than or equal to that with $\pi_{ij}=p$ for all $(i,j) \in \A_{\rm{C}}$.
			This, together with  Lemma \ref{theorem: strongLB} and the fact that each live-arc graph $\mathcal{G}^\omega_{\rm{C}}$ is constructed independently, 
			shows that the probability that all live-arc graphs $\mathcal{G}_{\rm{C}}^\omega$, $\omega \in \Omega$, are strongly connected is at least $p^*$ (defined in \eqref{defpstar}).
			The strong connectivity of graph $\mathcal{G}_{\rm{C}}^\omega$ implies that the number of SCCs in $\mathcal{G}_{\rm{C}}^\omega$ is one and $\mathcal{R}(\G_{\rm{C}}^{\omega},i)=\V_{\rm{C}}$ for all $i \in \V_{\rm{C}}$.
			As a result, with a probability at least $p^*$, (i) the number of variables $\boldsymbol{z}$ in the reduced formulation \eqref{strongIMP} (after applying the SCNA on \eqref{IMP}) is equal to $|\Omega|$; and (ii) the number of variables $\boldsymbol{z}$ in the reduced formulation \eqref{IMP_INA} (after applying the INA on \eqref{strongIMP}) is equal to one.
			\revv{This, together with the fact that there is a cardinality constraint \eqref{budgetcons} in \eqref{IMP_INA}, completes the proof.}
		\end{proof}
		The value $p^*$ in Theorem \ref{ERCorollary} can tend to one.
		For instance,  if $p \in(0,1]$ is a constant, then $p^* \rightarrow  1$ as $n \rightarrow \infty$ (notice that in this case, condition \eqref{assumptionSCC} also holds).
		This shows that after applying the SCNA and INA,  there are only $n+1$ variables and two linear constraints in the reduced SMCLP formulation \eqref{IMP_INA} with a probability tending to one. 
		
		\rev{It is worthwhile remarking that Theorem \ref{ERCorollary} also sheds a useful insight that for the IMP  with a general large and \revv{well-connected} network $\G$ 
			(not necessary to be complete) with high arc probabilities $\pi_{ij}$, 
			the SCNA and INA can be expected to effectively reduce the sizes of the live-arc graphs and SMCLP formulation \eqref{IMP}.
			Indeed, a \revv{well-connected} network is likely to contain \revv{large} complete subgraphs. 
			By Theorem \ref{ERCorollary}, with high arc probabilities, the nodes in these complete subgraphs are likely to be strongly connected in all live-arc graphs $\G^{\omega},~\omega\in\Omega$, 
			and as a result, more reductions are likely to be detected.
			This is consistent with the computational results in \revv{Section} \ref{sect: EICM} where more reductions can be detected by the proposed SCNA and INA 
			for large and \revv{well-connected} networks with large arc probabilities 
			(see Table \ref{table: RC_ICM} further ahead).}
	
	\section{The algorithms}\label{sect: al}
	In this section, we first discuss the implementation of the INA in \revv{Section} \ref{subsect: im_pre}.
	Then we present the BD algorithm with the proposed SCNA and INA for solving the IMP in \revv{Section} \ref{sect: bdi}.
	
	\subsection{An algorithm for identifying presolving reductions by the INA}\label{subsect: im_pre}
	
	A straightforward implementation of the INA requires to first precompute and store the reachability sets of all nodes in all live-arc graphs $\mathcal{G}^\omega,~\omega\in\Omega$, and then detect all 4-tuples $(\omega,\eta,i_1,i_2 )$ that satisfy condition \eqref{SCequalcondi_equation}.
	However, this leads to a high runtime complexity and a large memory consumption, which are  $\mathcal{O}({|\Omega|}^2{|\V|}^3)$ and $\mathcal{O}(|\Omega||\V|^2)$, respectively.
	In this subsection, we shall overcome this weakness by presenting a hashing-based heuristic algorithm.

	%
	%
	We first discuss the computation of the reachability sets.
	G\"{u}ney et al. \cite{Guney2020} computed the reachability set $\mathcal{R}(\mathcal{{G}}^\omega, i)$ of each node $i$ in each live-arc graph $\mathcal{G}^\omega$, $\omega\in\Omega$, by applying a reverse breadth-first search (BFS) starting from node $i$.
	The computational complexity is $\mathcal{O}(|\mathcal{V}|+|\mathcal{A}^\omega|)$.
	Here we notice that it can be (possibly) much faster to compute $\mathcal{R}(\mathcal{{G}}^\omega, i)$ based on the compact graph $\mathcal{\bar{G}}^\omega$.
	Indeed, as it has been mentioned in \revv{Section} \ref{sect: SCNA}, the reachability sets of nodes inside a given SCC $\mathcal{SC}_u^\omega$ of $\G^{\omega}$ are identical, where $u$ is the corresponding node in graph $\mathcal{\bar{G}}^\omega$.
	Hence, to compute the reachability sets of the nodes in SCC $\mathcal{SC}_u^\omega$, we only need to compute the reachability set $\mathcal{R}(\mathcal{G}^\omega, \mathcal{SC}_u^\omega)$.
	To compute the latter one, we can apply a reverse BFS in the compact graph $\mathcal{\bar{G}}^\omega$
	and use relation \eqref{relation_reachset}.
	The related complexity is $\mathcal{O}(|\mathcal{V}|+|\mathcal{\bar{A}}^\omega|)$, which  is potentially much smaller than $\mathcal{O}(|\mathcal{V}|+|\mathcal{A}^\omega|)$, especially when $|\mathcal{\bar{A}}^\omega|$ is much smaller than $|\mathcal{{A}}^\omega|$.
	Furthermore, as $\bar{\mathcal{G}}^{\omega}$ is a directed acyclic graph (as each node forms an SCC),
	we can perform a \emph{topological ordering} to further speed up the procedure of computing the reachability sets of all nodes in $\bar{\mathcal{G}}^{\omega}$.
	To be more specific, topological ordering for the directed acyclic graph $\bar{\mathcal{G}}^{\omega}$ is a linear ordering of nodes such that for each arc $(u_1,u_2) \in \bar{\mathcal{A}}^{\omega}$, node $u_1$ comes before node $u_2$ in the ordering.
	In our implementation, we traverse all nodes in $\bar{\mathcal{G}}^{\omega}$ to compute their reachability sets according to the {topological ordering}.
	In other words, when computing the reachability set of node $u\in\bar{\V}^{\omega}$, the reachability sets of nodes $u'$, $u'\in \mathcal{N}_{\omega}^-(u)$, have been computed, where $ \mathcal{N}_{\omega}^-(u):=\{ u' \,:\,(u',u) \in \mathcal{\bar{A}}^\omega \}$ is the set of node $u$'s \revv{incoming} neighbor nodes.
	Therefore, to compute node $u$'s reachability set,  we only need to traverse its
	\revv{incoming} neighbor nodes and use relation $ \mathcal{R}(\mathcal{\bar{G}^\omega}, u) = \bigcup_{u' \in \mathcal{N}_{\omega}^-(u)} \mathcal{R}(\mathcal{\bar{G}^\omega}, u')\cup \{u\}$.
	This avoids performing a whole reverse BFS and generally accelerates the computation of $\mathcal{R}(\mathcal{\bar{G}^\omega}, u)$.

	Next, we discuss the storage of the reachability sets.
	First, we can implement SCNA to alleviate the memory consumption of storing the reachability sets (as for each SCC, only a single node's reachability set needs to be stored).
	Second, to further avoid a large memory consumption,  we \revv{only store} those nodes' reachability sets whose sizes are smaller than or equal to a predefined parameter $\maxlength$.
	This also means that only nodes satisfying this criterion will be used for detecting the reductions by the INA.
	%
	%
	The rationale behind this strategy is that for the nodes with smaller reachability sets, it is more likely to detect nodes that are isomorphic to them, as illustrated in our computational results (see \revv{Section} \ref{subsec:parameter} further ahead).

	Finally, we apply the INA by detecting the pairs whose reachability sets are identical, i.e.,
	\begin{equation}\label{SCequalcondi_equationeq}
		\bigcup_{v\in\mathcal{R}(\bar{\mathcal{G}}^{\omega},u)}\mathcal{SC}^{\omega}_v=\bigcup_{v\in\mathcal{R}(\bar{\mathcal{G}}^{\revv{\eta}},u_0)}\mathcal{SC}^{\revv{\eta}}_v
	\end{equation}
	for some $u\in \mathcal{\bar{V}}^\omega$ and $ u_0 \in \mathcal{\bar{V}}^{\revv{\eta}}$.
	To do this, we follow \cite{Gurobipresolve} to use a hashing-based method.
	The basic idea of the hashing-based method is to simultaneously build a hashing table that remembers the information of the reachability sets\revv{,} and test whether there exists a reachability set in the hashing table that is identical to the one we are currently looking at.
	Specifically, let $\mathcal{H}$ be the hashing table and for each $\omega\in\Omega$ and $u\in\bar{\mathcal{V}}^{\omega}$,
	let
	$\bigcup_{v\in\mathcal{R}(\bar{\mathcal{G}}^{\omega},u)}\mathcal{SC}^{\omega}_v$
	be the key with 3-tuple $(\omega,u,\mathcal{R}(\bar{\mathcal{G}}^{\omega},u))$ being the stored value in table $\mathcal{H}$.
	At first, table $\mathcal{H}$ is initialized to be $\varnothing$.
	Then, in each iteration, for scenario $\omega\in\Omega$ and node $u\in\bar{\mathcal{V}}^{\omega}$ (with $|\bigcup_{v\in\mathcal{R}(\bar{\mathcal{G}}^{\omega},u)}\mathcal{SC}^{\omega}_v| \leq \maxlength$), table  $\mathcal{H}$ is queried for $\bigcup_{v\in\mathcal{R}(\bar{\mathcal{G}}^{\omega},u)}\mathcal{SC}^{\omega}_v$.
	If condition \eqref{SCequalcondi_equationeq} holds for some corresponding entry $(\revv{\eta}, u_0, \mathcal{R}(\bar{\mathcal{G}}^{\revv{\eta}},u_0))$ in table $\mathcal{H}$, we apply the INA by removing variable $z_u^\omega$, deleting the associated constraint in \eqref{strongconnectioncons}, and adding the objective coefficient of variable $z_u^\omega$ into that of variable $z_{u_0}^{\revv{\eta}}$; otherwise,  tuple $(\omega,u,\mathcal{R}(\bar{\mathcal{G}}^{\omega},u))$ will be added into table $\mathcal{H}$.
	The procedure is repeated until all considered reachability sets are tested.
	
	In summary, we present the implementation of the INA in Algorithm \ref{algorithm_INA} to obtain the reduced SMCLP formulation \eqref{IMP_INA}.
	Here, for simplicity of presentation, the improvement of topological ordering is omitted in Algorithm \ref{algorithm_INA}.
	%
	%
	In step 4, we perform a reverse BFS on node $u$ in graph $\mathcal{\bar{G}}^{\omega}$ to compute $\mathcal{R}(\bar{\mathcal{G}}^{\omega},u)$.
	In steps 5-9, we use the hashing-based method to detect whether there exists some entry $(\revv{\eta},u_0,\mathcal{R}(\bar{\mathcal{G}}^{\revv{\eta}},u_0))$ in table $\mathcal{H}$ such that \eqref{SCequalcondi_equationeq} holds.
	If yes, we apply the INA reductions; otherwise, we add the new entry $(\omega,u,\mathcal{R}(\bar{\mathcal{G}}^{\omega},u))$   into table $\mathcal{H}$.

	\begin{algorithm}[h]
		\setstretch{1}
		\caption{Implementation of the INA}\label{algorithm_INA}
		\KwIn{ the compact live-arc graphs $\bar{\mathcal{G}}^{\omega}=(\bar{\mathcal{V}}^{\omega},\bar{\mathcal{A}}^{\omega})$ and the SCCs $\mathcal{SC}_u^\omega$, $u \in \bar{\mathcal{V}}^{\omega}$, of the (original) live-arc graphs  $\mathcal{G}^\omega=(\mathcal{V},{\mathcal{A}}^{\omega})$, $\omega\in\Omega$.}
		\KwOut{sets $\tilde{\V}^{\omega}$ and the objective coefficients $f_u^{\omega}$ of variables $z_{u}^\omega$,  $\omega\in\Omega$, $ u\in\tilde{\V}^{\omega}$, in the reduced formulation \eqref{IMP_INA}.}
		Initialize $\tilde{\V}^{\omega}:=\bar{\V}^{\omega}$, $f_u^{\omega}:=p^{\omega}|\mathcal{SC}_u^{\omega}|$ for all $\omega\in\Omega$ and $u\in\bar{\mathcal{V}}^{\omega}$, and $\mathcal{H}:=\varnothing$\;
		\For{$\omega\in\Omega$}
		{
			\For{$u\in\bar{\mathcal{V}}^{\omega}$}
			{
				Perform a reverse BFS on node $u$ in graph $\bar{\mathcal{G}}^{\omega}$ to compute $\mathcal{R}(\bar{\mathcal{G}}^{\omega},u)$\;
				\eIf{$|\bigcup_{v\in\mathcal{R}(\bar{\mathcal{G}}^{\omega},u)}\mathcal{SC}^{\omega}_v|\leq \maxlength$ and  \eqref{SCequalcondi_equationeq} holds for some $(\revv{\eta},u_0,\mathcal{R}(\bar{\mathcal{G}}^{\revv{\eta}},u_0))\in\mathcal{H}$}
				{
					Set $f_{u_0}^{\revv{\eta}}:=f_{u_0}^{\revv{\eta}}+f_u^{\omega}$ and $\tilde{\V}^{\omega}:=\tilde{\V}^{\omega}\backslash\{u\}$\;
				}
				{
					$\mathcal{H}:=\mathcal{H}\cup\left\{(\omega,u,\mathcal{R}(\bar{\mathcal{G}}^{\omega},u))\right\}$\;
				}
			}
		}
	\end{algorithm}

	\subsection{\revv{BD} algorithm for solving the IMP}\label{sect: bdi}
	
	G\"{u}ney et al. \cite{Guney2020} has proposed the BD algorithm to solve the IMP based on formulation \eqref{IMP}.
	The authors showed the effectiveness of integrating the presolving method SNA into the BD algorithm.
	In this subsection, to further enhance the capability of using the BD algorithm to solve the IMP, we attempt to integrate the proposed SCNA and INA into the BD algorithm, or equivalently, to design a BD algorithm that is based on the reduced formulation \eqref{IMP_INA} and the compact graphs $\bar{\mathcal{G}}^\omega$, $\omega \in \Omega$.
	
	
	\subsubsection{Reformulation of \eqref{IMP_INA}}
	\label{BDref}
	We first briefly introduce the BD reformulation of \eqref{IMP_INA} (more details can be found in \cite{Guney2020}).
	To begin with, we note that replacing each binary variable $z_u^\omega$ by a continuous variable taking value in $[0,1]$ does not change the optimal value of formulation \eqref{IMP_INA}.
	For each $\omega\in\Omega$, let $\revv{\varphi^\omega}$ represent \revv{the variable} that captures the contribution of scenario $\omega$ to the objective function.
	Then, we can project out variables $\boldsymbol{z}$ and equivalently reformulate \eqref{IMP_INA} as
	\begin{equation}\label{IMP2}
		\max_{\boldsymbol{y},\,\boldsymbol{\varphi}} \left\{\sum_{\omega\in\Omega}\varphi^{\omega}\,:\, 	\eqref{budgetcons},~\eqref{ybincons},~\varphi^{\omega}\leq\Phi^{\omega}(\boldsymbol{y}),\forall~ \omega\in\Omega\right\}
	\end{equation}
	where function $\Phi^{\omega}(\boldsymbol{y})$ is defined as follows:
	\begin{equation}\label{subIMP}
		\Phi^{\omega}(\boldsymbol{y}):=\max_{\boldsymbol{z}}\left\{\sum_{u\in\tilde{\mathcal{V}}^{\omega}}f_u^{\omega}z_u^{\omega}\,:\,
		z_u^{\omega}\leq\sum_{v\in\mathcal{R}(\bar{\mathcal{G}}^{\omega},u)}y(\mathcal{SC}_v^{\omega}),~0\leq z_u^{\omega}\leq 1,~\forall~ u\in\tilde{\mathcal{V}}^{\omega}\right\}.
	\end{equation}
	For a fixed $\bar{\boldsymbol{y}} \in {[0,1]}^{|\mathcal{V}|}$, to model the inequalities $\varphi^{\omega}\leq\Phi^{\omega}(\boldsymbol{y})$ for all $\omega\in\Omega$, we use the Benders optimality cuts which are derived as follows.
	First, the dual of formulation \eqref{subIMP} when $\boldsymbol{y}=\bar{\boldsymbol{y}}$ is
	\begin{equation}\label{dualsubIMP_fix_y}
		\min_{\boldsymbol{\alpha}^\omega,\,\boldsymbol{\beta}^\omega  }\left\{\sum_{u\in\tilde{\mathcal{V}}^{\omega}}\left(\alpha_u^{\omega}\sum_{v\in\mathcal{R}(\bar{\mathcal{G}}^{\omega},u)}\bar{y}(\mathcal{SC}_v^{\omega})+\beta_u^{\omega}\right)\,:\,\alpha_u^{\omega}+\beta_u^{\omega}\geq f_u^{\omega},~\alpha_u^{\omega},~\beta_u^{\omega}\geq0,~\forall~ u\in\tilde{\mathcal{V}}^{\omega}\right\},
	\end{equation}
	where $\alpha_u^{\omega}$ and $\beta_u^{\omega}$ are the dual variables of constraints $z_u^{\omega}\leq\sum_{v\in\mathcal{R}(\bar{\mathcal{G}}^{\omega},u)}\bar{y}(\mathcal{SC}_v^{\omega})$ and $z_u^{\omega}\leq 1$, respectively.
	%
	Clearly, formulation \eqref{dualsubIMP_fix_y} has a closed form solution \revv{$(\tilde{\boldsymbol{{\alpha}}}^{\omega},\tilde{\boldsymbol{{\beta}}}^{\omega})$:
		\begin{equation}\label{ab}
			\begin{aligned}
				(\tilde{\alpha}_u^{\omega},\tilde{\beta}_u^{\omega})=\left\{
				\begin{array}{ll}
					(0,f_u^{\omega}),& {\text{if}}~\sum_{v\in\mathcal{R}(\bar{\mathcal{G}}^{\omega},u)}\bar{y}(\mathcal{SC}_v^{\omega})\geq 1;\\[3pt]
					(f_u^{\omega},0),& \text{otherwise},\\
				\end{array}\right. ~  \forall ~u \in \tilde{\mathcal{V}}^{\omega}.
			\end{aligned}
	\end{equation}}%
	Then the Benders optimality cuts for formulation \eqref{IMP2} are given by
	\revv{\begin{equation}\label{benderscut1}
			\varphi^{\omega}\leq\sum_{u\in\tilde{\mathcal{V}}^{\omega}}\left(\tilde{\alpha}_u^{\omega}\sum_{v\in\mathcal{R}(\bar{\mathcal{G}}^{\omega},u)}y(\mathcal{SC}_v^{\omega})+\tilde{\beta}_u^{\omega}\right),~\forall~ \omega\in\Omega.
		\end{equation}
		From \eqref{ab},  $(\tilde{\boldsymbol{{\alpha}}}^{\omega},\tilde{\boldsymbol{{\beta}}}^{\omega})$  depends on $\bar{\boldsymbol{y}}$.
		Thus, inequality \eqref{benderscut1} depends on $\bar{\boldsymbol{y}}$ as well, but we omit this dependence for notational convenience.}
	
	\revv{To solve formulation \eqref{IMP2}, we use a branch-and-Benders-cut approach in which a branch-and-cut search tree is created and the Benders optimality cuts \eqref{benderscut1} are separated at each branch-and-cut node.
		Following \cite{Guney2020,Wu2017}, we start with a relaxed master problem of  \eqref{IMP2} in which inequalities \eqref{benderscut1} with $\bar{\boldsymbol{y}}=\boldsymbol{0}$, i.e.,
		\begin{equation}\label{subineq}
			\varphi^{\omega}\leq\sum_{u\in\tilde{\mathcal{V}}^{\omega}}f_u^\omega\sum_{v\in\mathcal{R}(\bar{\mathcal{G}}^{\omega},u)}y(\mathcal{SC}_v^{\omega}),~\forall ~\omega\in\Omega,
		\end{equation}
		are added. Here $\boldsymbol{0}$ is the $|\mathcal{V}|$-dimensional zero vector.}

	For a given point $\bar{\boldsymbol{y}} \in {[0,1]}^{|\mathcal{V}|}$, it is interesting to ask whether or not applying the SCNA or INA changes the Benders optimality cuts \eqref{benderscut1}, or equivalently, whether or not the Benders optimality cuts \revv{\eqref{bc1}, \eqref{bc4}, and \eqref{benderscut1}} based on formulations \eqref{IMP}, \eqref{strongIMP}, and \eqref{IMP_INA} are equivalent\revv{, where}
	\revv{
		\begin{align}
			& \varphi^{\omega}\leq\sum_{i\in\mathcal{V}}\left(\hat{\alpha}_i^{\omega}\sum_{j\in\mathcal{R}(\mathcal{G}^{\omega},i)}y_j+\hat{\beta}_i^{\omega}\right),~\forall~\omega \in \Omega,\label{bc1}\\
			& \varphi^{\omega}\leq\sum_{u\in\bar{\mathcal{V}}^{\omega}}\left(\bar{\alpha}_u^{\omega}\sum_{v\in\mathcal{R}(\bar{\mathcal{G}}^{\omega},u)}y(\mathcal{SC}_v^{\omega})+\bar{\beta}_u^{\omega}\right),~\forall~\omega \in \Omega.	\label{bc4}
		\end{align}
	}%
	\revv{Here,
		\begin{align}\label{ab1}
			& (\hat{\alpha}_i^{\omega},\hat{\beta}_i^{\omega})=\left\{
			\begin{array}{ll}
				(0,p^\omega),& {\text{if}}~\sum_{j\in\mathcal{R}(\mathcal{G}^{\omega},i)}\bar{y}_j\geq 1;\\[3pt]
				(p^\omega,0),& \text{otherwise},\\
			\end{array}\right. ~  \forall ~i \in {\mathcal{V}},\\
			&(\bar{\alpha}_u^{\omega},\bar{\beta}_u^{\omega})=\left\{
			\begin{array}{ll}
				(0,p^\omega|\mathcal{SC}_u^{\omega}|),& {\text{if}}~\sum_{v\in\mathcal{R}(\bar{\mathcal{G}}^{\omega},u)}\bar{y}(\mathcal{SC}_v^{\omega})\geq 1;\\[3pt]
				(p^\omega|\mathcal{SC}_u^{\omega}|,0),& \text{otherwise},\\
			\end{array}\right. ~  \forall ~u \in \bar{\mathcal{V}}^{\omega}.\label{ab2}
		\end{align}
	}%
	This question is addressed by the following proposition.
	\begin{proposition}\label{bcnotchanged}
		\revv{Given a point $\bar{\boldsymbol{y}} \in {[0,1]}^{|\mathcal{V}|}$ and a scenario $\omega \in \Omega$, (i) \eqref{bc1} and \eqref{bc4} are equivalent; (ii) \eqref{benderscut1} and \eqref{bc4} may be different.}
	\end{proposition}
	\begin{proof}
		The proof \revv{is given in} Appendix \ref{proof_scna_BD}.
	\end{proof}
	
	\revv{By Proposition \ref{bcnotchanged} (i), we know that Benders optimality cuts before and after applying the SCNA are identical (if INA is not applied).
		As a result, the SCNA will not change the path of the search tree created by the branch-and-Benders-cut algorithm.}
	\revv{However, the separation of the Benders optimality cuts \eqref{bc4} after applying the SCNA could be much more efficient than that of the Benders optimality cuts \eqref{bc1} before applying the SCNA.
		Indeed, given a scenario $\omega$, the separation of \eqref{bc4} can be done in $\mathcal{O}(|\mathcal{V}|+\sum_{u \in \bar{\mathcal{V}}^\omega} |\mathcal{R}(\bar{\mathcal{G}}^\omega, u)|)\approx \mathcal{O}(\sum_{u \in \bar{\mathcal{V}}^\omega} |\mathcal{R}(\bar{\mathcal{G}}^\omega, u)|)$ (by computing $\bar{y}(\mathcal{SC}_v^\omega)$ for all $v \in \bar{\mathcal{V}}^\omega$ in $\mathcal{O}(|\mathcal{V}|)$ and \eqref{ab} in $\mathcal{O}(\sum_{u \in \bar{\mathcal{V}}^\omega} |\mathcal{R}(\bar{\mathcal{G}}^\omega, u)|)$.}
	\revv{This could be much smaller than the complexity of a direct implementation of the separation of Benders optimality cuts \eqref{bc1}, which is  $\mathcal{O}(\sum_{i \in {\mathcal{V}}} |\mathcal{R}({\mathcal{G}}^\omega, i)|)$, especially when
		\begin{itemize}
			\item [(a)] $\mathcal{G}^\omega$'s number of SCCs is much smaller than its number of nodes, or
			\item [(b)] the number of elements in the reachability sets of the nodes in $\bar{\mathcal{G}}^\omega$ are much smaller than that in $\mathcal{G}^\omega$.
	\end{itemize}}%
	\revv{\noindent Different from that of applying the SCNA, applying the INA can generally change the Benders optimality cuts, as stated in Proposition \ref{bcnotchanged} (ii).
		This means that the path of the search tree created by the branch-and-Benders-cut algorithm could also be different.
		Nevertheless, as $|\tilde{\mathcal{V}}^\omega|$ may be smaller than  $|\bar{\mathcal{V}}^\omega|$, applying the INA can further improve the separation of the Benders optimality cuts.}
	\revv{We remark that as stated in \cite{Guney2020}, efficient separation of Benders optimality cuts is crucial for the BD algorithm to successfully solve large-scale IMPs.
		In the following, we shall present a more efficient separation algorithm for the Benders optimality cuts by integrating the proposed SCNA and INA into the one developed in \cite{Guney2020}.}

	\subsubsection{\revv{An efficient separation algorithm}}
	
	Having a solution \revv{$(\bar{\boldsymbol{y}},\bar{\boldsymbol{\varphi}}) \in [0,1]^{|\mathcal{V}|} \times \mathbb{R}_+^{|\Omega|}$} of the LP relaxation of this relaxed master problem \revv{of \eqref{IMP2}},
	\revv{we now describe an efficient separation algorithm for} Benders optimality cuts \eqref{benderscut1}.
	\revv{As mentioned in Section \ref{subsect: im_pre}}, it is unrealistic to compute and store all reachability sets of all nodes of all scenarios a priori due to the large memory consumption.
	Hence, similar to \cite{Guney2020}, we introduce parameter $\mathbf{MemLimPerScen}$ to denote the maximally allowed memory consumption per scenario.
	In particular, for each scenario $\omega\in\Omega$, we store the reachability sets of nodes according to their topological ordering in the compact live-arc graph $\bar{\G}^{\omega}$ until the memory consumption reaches $\mathbf{MemLimPerScen}$.
	As a result, when computing Benders optimality cut \eqref{benderscut1}, if $\mathcal{R}(\bar{\mathcal{G}}^{\omega},u)$ has been stored, we access it directly;
	otherwise, we perform a reverse BFS to compute $\mathcal{R}(\bar{\mathcal{G}}^{\omega},u)$ on the fly.
	Moreover, to further improve the efficiency of computing Benders optimality cut \eqref{benderscut1}, we can omit to compute $\mathcal{R}(\bar{\mathcal{G}}^{\omega},u)$ if $\revv{(\tilde{\alpha}_u^{\omega},\tilde{\beta}_u^{\omega})}= (0, f_u^\omega)$ is known a priori.
	More specifically, let $\mathcal{S}^{\omega}:=\{u\in\bar{\mathcal{V}}^{\omega}\,:\,\bar{y}(\mathcal{SC}_u^{\omega})
	\geq1\}$ and \revv{$\revv{\mathcal{I}(\bar{\mathcal{G}}^\omega, \mathcal{S}^{\omega})}$ denote the nodes in $ \tilde{\mathcal{V}}^\omega$ that can be influenced by some node in $\mathcal{S}^\omega$, i.e., $\revv{\mathcal{I}(\bar{\mathcal{G}}^\omega, \mathcal{S}^{\omega})}=\left \{u\in\tilde{\mathcal{V}}^{\omega}\,:\,  \text{there exists~a~directed~path~in~$\bar{\mathcal{G}}^\omega$~from~a~node~in}~\mathcal{S}^{\omega}~\text{to~node}~u\right\}$}.
	Then, for each $u\in \revv{\mathcal{I}(\bar{\mathcal{G}}^\omega, \mathcal{S}^{\omega})}$, we must have  $\sum_{v\in\mathcal{R}(\bar{\mathcal{G}}^{\omega},u)}\bar{y}(\mathcal{SC}_v^{\omega})
	\geq1$, and hence $\revv{(\tilde{\alpha}_u^{\omega},\tilde{\beta}_u^{\omega})}= (0, f_u^\omega)$.
	In summary, we present the separation of Benders optimality cuts in Algorithm \ref{algorithm_benders}.
	
	\begin{algorithm}[h]
		\caption{Separation of Benders optimality cuts \eqref{benderscut1}}\label{algorithm_benders}
		\KwIn{the compact live-arc graphs $\bar{\mathcal{G}}^{\omega}=(\bar{\mathcal{V}}^{\omega},\bar{\mathcal{A}}^{\omega})$, $\omega \in \Omega$, formulation \eqref{IMP_INA}, and point \revv{$(\bar{\boldsymbol{y}},\bar{\boldsymbol{\varphi}}) \in [0,1]^{|\mathcal{V}|} \times \mathbb{R}_+^{|\Omega|}$}.
		}
		\KwOut{the set $\mathcal{C}$ of Benders optimality cuts which are violated by point $(\bar{\boldsymbol{y}},\bar{\boldsymbol{\varphi}})$.}
		Initialize $\mathcal{C}:=\varnothing$\;
		\For{$\omega\in\Omega$}
		{
			Compute $\mathcal{S}^{\omega}:=\left \{u\in\bar{\mathcal{V}}^{\omega}\,:\, \bar{y}(\mathcal{SC}_u^{\omega})
			\geq1\right \}$\;
			Perform a BFS in $\bar{\G}^{\omega}$ to compute $\revv{\mathcal{I}(\bar{\mathcal{G}}^\omega, \mathcal{S}^{\omega})}$\;
			Initialize $C^{\omega}:=\sum_{u\in\revv{\mathcal{I}(\bar{\mathcal{G}}^\omega, \mathcal{S}^{\omega})}}f_u^{\omega}$ and $c_j^{\omega}:=0$, $j\in\mathcal{V}$\;
			\For{$u\in\tilde{\mathcal{V}}^{\omega}\backslash\revv{\mathcal{I}(\bar{\mathcal{G}}^\omega, \mathcal{S}^{\omega})}$}
			{
				\If{$\mathcal{R}(\mathcal{\bar{G}}^\omega, u)$ is not stored in the memory}{
					Perform a reverse BFS in $\bar{\G}^{\omega}$ to compute $\mathcal{R}(\mathcal{\bar{G}}^\omega, u)$\;}
				\eIf{$\sum_{v\in\mathcal{R}(\bar{\mathcal{G}}^{\omega},u)}\bar{y}(\mathcal{SC}_v^{\omega})\geq1$}
				{
					Set $C^{\omega}:=C^{\omega}+f_u^{\omega}$\;
				}
				{
					Set $c_j^{\omega}:=c_j^{\omega}+f_u^{\omega}$ for all $ j \in \bigcup_{v\in\mathcal{R}(\bar{\mathcal{G}}^{\omega},u)}\mathcal{SC}_v^{\omega} $\;
				}
			}
			\If{$\bar{\varphi}^{\omega}>\sum_{j\in\mathcal{V}}c_j^{\omega}\bar{y}_j+C^{\omega}$}
			{
				Set $\mathcal{C}:=\mathcal{C}\cup\left \{\varphi^{\omega}\leq \sum_{j\in\mathcal{V}}c_j^{\omega}y_j+C^{\omega}\right \}$\;
			}
		}
	\end{algorithm}
	\revv{In Algorithm \ref{algorithm_benders},  the Benders optimality cuts \eqref{benderscut1} are rewritten as
		\begin{equation}
			\label{benderscut2}
			\varphi^{\omega}\leq\sum_{j\in\mathcal{V}}c_j^{\omega}y_j +C^{\omega},~\forall~ \omega\in\Omega,
		\end{equation}
		where $C^{\omega}:=\sum_{u\in\tilde{\mathcal{V}}^{\omega}}\tilde{\beta}_u^{\omega}$, $c_j^{\omega}:=\sum_{u\in\mathcal{J}_j^{\omega}}\tilde{\alpha}_u^{\omega}$, $j \in\mathcal{V}$, and $\mathcal{J}_j^{\omega}:=\left\{u\in\tilde{\mathcal{V}}^{\omega}\,:\, j\in\bigcup_{v\in\mathcal{R}(\bar{\mathcal{G}}^{\omega},u)}\mathcal{SC}_v^{\omega}\right\}$.
		In other words,} $C^\omega$ and $c_j^\omega$ for each $j \in \V$ are used to keep track of the constant term and the coefficient of variable $y_j$ in Benders optimality cut \eqref{benderscut1}, respectively.
	For each $\omega \in \Omega$, we initialize $C^\omega$ and $c_j^\omega$ in step 5 and then sequentially update them \revv{in steps 10-14} depending on whether or not $\sum_{v\in\mathcal{R}(\bar{\mathcal{G}}^{\omega},u)}\bar{y}(\mathcal{SC}_v^{\omega})
	\geq1$ \revv{holds}.
	Finally, in steps 16-18, if Benders optimality cut $\varphi^{\omega}\leq \sum_{j\in\mathcal{V}}c_j^{\omega}y_j+C^{\omega}$ is violated by point $(\bar{\boldsymbol{y}},\bar{\boldsymbol{\varphi}})$, we add it into the set of violated Benders optimality cuts $\mathcal{C}$.
	
	It is worth emphasizing the computational efficiency of our separation algorithm for the Benders optimality cuts over the one in \cite{Guney2020}.
	First, in Algorithm \ref{algorithm_benders},  to compute sets  $\revv{\mathcal{I}(\bar{\mathcal{G}}^\omega, \mathcal{S}^{\omega})}$ and  $\mathcal{R}(\mathcal{\bar{G}}^\omega, u)$, we perform (reverse) BFSes in the compact live-arc graph $\bar{\mathcal{G}}^{\omega}$,
	which is potentially much faster than that in \cite{Guney2020} where the (reverse) BFSes in the (original) live-arc graph ${\mathcal{G}}^{\omega}$ are performed.
	\revv{Second, compared with that in \cite{Guney2020}, fewer reverse BFSes will be performed in Algorithm \ref{algorithm_benders} due to the following two reasons. }
	\begin{itemize}
		\item [(i)] \revv{$|\tilde{\mathcal{V}}^\omega\backslash \mathcal{I}(\bar{\mathcal{G}}^\omega, \mathcal{S}^\omega)|$ in step 6 of Algorithm \ref{algorithm_benders} can be potentially much smaller than that in \cite{Guney2020}, which is $|{\mathcal{V}}\backslash \mathcal{I}({\mathcal{G}}^\omega, \hat{\mathcal{S}}^\omega)|$.
			Here $\hat{\mathcal{S}}^\omega=\{j \in \mathcal{V}\, : \, \bar{y}_j =1\}$, and $\revv{\mathcal{I}({\mathcal{G}}^\omega, \hat{\mathcal{S}}^{\omega})}$ denotes the set of nodes in $\mathcal{V}$ that can be influenced by some node in $\hat{\mathcal{S}}^\omega$ in graph $\mathcal{G}^\omega$.
			Indeed, a node $u\in \tilde{\mathcal{V}}^\omega\backslash \mathcal{I}(\bar{\mathcal{G}}^\omega, \mathcal{S}^\omega)$ corresponds to an SCC in $\mathcal{G}^\omega$ whose nodes cannot be influenced by the nodes in $\bigcup_{v \in \mathcal{S}^\omega}\mathcal{SC}_v^\omega$.
			As a result,
			\begin{equation}
				|\tilde{\mathcal{V}}^\omega\backslash \mathcal{I}(\bar{\mathcal{G}}^\omega, \mathcal{S}^\omega)|
				\leq \left|\mathcal{V}\backslash \mathcal{I}\left(\mathcal{G}^\omega,\bigcup_{v \in \mathcal{S}^\omega}\mathcal{SC}_v^\omega\right) \right|
				\leq |{\mathcal{V}}\backslash \mathcal{I}({\mathcal{G}}^\omega, \hat{\mathcal{S}}^\omega)|
			\end{equation}
			where the last inequality follows from $\hat{\mathcal{S}}^\omega \subseteq \bigcup_{v \in \mathcal{S}^\omega}\mathcal{SC}_v^\omega$ and the fact that $\mathcal{I}(\mathcal{G}^\omega, \mathcal{S}_1) \subseteq \mathcal{I}(\mathcal{G}^\omega, \mathcal{S}_2)$ for any $ \mathcal{S}_1 \subseteq  \mathcal{S}_2 \subseteq \mathcal{V}$.}
		\item [(ii)] \revv{As stated in Section \ref{BDref}, the sum of the sizes of the reachability sets $\mathcal{R}(\bar{\mathcal{G}}^\omega,u)$, $u \in \tilde{\mathcal{V}}^\omega$, could be much smaller than that of the sizes of the reachability sets $\mathcal{R}({\mathcal{G}}^\omega,i)$, $i \in \mathcal{V}$.
			This implies that for a fixed $\mathbf{MemLimPerScen}$, more reachability sets $\mathcal{R}(\bar{\mathcal{G}}^\omega,u)$ are likely to be stored a priori, and thus the condition in step 7 of Algorithm \ref{algorithm_benders} is less likely to occur.  }
	\end{itemize}

	\section{Computational results}\label{sect: CR}
	In this section, we present the computational results to show the effectiveness of the proposed SCNA and INA.
	We use the BD algorithm in \revv{Section} \ref{sect: bdi}, which \revv{was} implemented in C++ linked with IBM ILOG CPLEX optimizer 20.1.0 \cite{Cplex}.
	The Benders \revv{optimality} cuts \revv{were} added using CALLABLE LIBRARIES under the default settings of the branch-and-cut framework of CPLEX.
	The time limit \revv{was} set to 14400 seconds, and all the experiments were performed on a cluster of Intel(R) Xeon(R) Gold 6140 CPU @ 2.30GHz computers.
	Only a single core was used in our experiments.
	\rev{We note here that throughout this section,
		all averages are taken to be geometric means.
		Since the statistics can be zero, we use the shifted geometric mean with a
		shift of $1$ (the shifted geometric mean of values $x_1, x_2,\ldots,x_n$ with shift $s$ is defined as
		$\prod_{k=1}^n(x_k + s)^{1/n} - s$; see \cite{Achterberg2007}).}
	%
	%
	\subsection{Networks and settings}
	\label{settings}
	Our benchmark data set consists of eight real-world social networks. Four of them have been used in \cite{Guney2020,Wu2017} (MSG, GNU, HEP, and ENRON) and the other four networks are from the SNAP database \footnote{\url{https://snap.stanford.edu/data/}.}
	(FACEBOOK, DEEZER, TWITTER, and EPINIONS).
	The latter ones are large-scale networks  which are used to test the performance of the SCNA and INA in large-scale cases.
	For the undirected networks (GNU, HEP, FACEBOOK, and DEEZER), we convert them into directed networks by adding two directed arcs $(i,j)$ and $(j,i)$ for each edge $(i,j)$.
	\revv{Table \ref{table: instances} summarizes the basic information of these networks, where $|\mathcal{V}|$,  $|\mathcal{A}|$, and $|\mathcal{A}'|$ denote the number of nodes, the number of arcs (including parallel arcs), and  the number of unique arcs, respectively.
		In addition, we also report the arithmetic mean of nodes' degrees  $\rho:=\frac{|\mathcal{A}|}{|\mathcal{V}|}$ \footnote{\revv{$\frac{|\mathcal{A}|}{|\mathcal{V}|}$ is equal to the arithmetic mean of nodes' indegees or outdegrees of the directed network.}}, which reflects the connectivity of the networks \cite{Beineke2002,Cheung2021}.
	}%
	
	\begin{table}[h]
		\footnotesize
		\centering
		\caption{Eight real-world social networks.}\label{table: instances}
		\setlength{\tabcolsep}{2pt}
		\renewcommand{\arraystretch}{1.5}
		\begin{tabular}{|lllll|l|}
			\hline
			{Network} & $|\mathcal{V}|$&{$|\mathcal{A}|$}&$|\mathcal{A}'|$&\makecell[l]{$\rho$}&{Description}\\
			\hline
			MSG&1899&59835& 20296& 31.5 &Messaging network of UC-Irvine \cite{MSG}\\
			GNU&10879&79988& 79988& 7.4 &Gnutella peer-to-peer file sharing network	 \cite{GNU}\\
			HEP&15233&117782& 64426& 7.7 &High energy physics paper citation network \cite{HEP}\\
			ENRON&36692&367662& 367662& 10.0 & Email communication network from Enron \cite{ENRON}\\
			FACEBOOK &50515&1638612&1638180 &32.4  &Facebook page network in the category of artist \cite{FACEBOOK}\\
			DEEZER & 54573 & 996404 & 996404& 18.3 & Deezer friendship network of users in Croatia \cite{FACEBOOK}\\
			TWITTER&81306&1768149& 1768135& 21.7 & Social network from Twitter \cite{TWITTER}\\
			EPINIONS &131828&841372& 840799& 6.4 & Who-trust-whom social network of Epinions \cite{EPI2}\\
			\hline
		\end{tabular}
	\end{table}

	The selections of parameters of the IMP are also similar to the one in \cite{Guney2020,Wu2017}.
	More specifically, for the IMP under both the ICM and LTM, the cardinality restriction $K$ in formulation \eqref{IMP} is selected in $\{5,10,15,25\}$.
	\revv{As stated at the end of Section \ref{sect:pd}, the number of scenarios $|\Omega|$ is a key parameter to achieve the trade-off between the approximation quality and solution time for solving the sampling version of the IMP.
		In Section 2 of \cite{Chen2022}, we follow \cite{Kahr2020} to conduct experiments to compare the  \emph{approximation gaps} (defined in equation (4) of \cite{Kahr2020}) of the sampling version of the IMP with different number of scenarios.
		The approximation gap estimates the difference of the objective value of the sampling version of the IMP and the true objective value of the IMP.
		The results showed  that $|\Omega|=1000$ is a reasonable choice to achieve a solution with a small approximation gap for instances constructed by the considered networks.
		Therefore, in our experiments,  to reflect different approximation levels of the sampling version
		of the IMP, the number of scenarios $|\Omega|$ is selected in $\{250,500,1000\}$.}
	We next discuss the selections of the activation probability $\pi_{ij}$ (under the ICM) and the weight $b_{ij}$ (under the LTM) for each arc $(i,j)$.
	Let $n_{ij}$ denote the number of parallel arcs from node $i$ to node $j$.
	For the IMP under the ICM, each single arc is assigned the same activation probability $p$ chosen in $\{0.01,0.05,0.10\}$.
	%
	\revv{For the IMP under the LTM,
		we set influence weight on arc $(i,j)$ as $1/n_j$, where $n_j:=\sum_{i\,:\,(i,j)\in\A}n_{ij}$ (\revv{i.e.}, the number of incoming arcs of node $j$ in $\G$) is a normalization factor to ensure that the sum of weights of the incoming arcs to $j$ is \revv{at most} 1.}
	\revv{Each live-arc graph $\mathcal{G}^\omega$ of a scenario $\omega\in \Omega$ is randomly constructed as follows.
		For a live-arc graph $\mathcal{G}^{\omega}$ under the ICM, arc $(i,j)\in \mathcal{A}$ is included in $\mathcal{A}^\omega$ with probability $\pi_{ij}=1-(1-p)^{n_{ij}}$ (representing that at least one of the parallel arcs of $(i,j)$
		appears in $\mathcal{G}^{\omega}$).
		For a live-arc graph $\mathcal{G}^{\omega}$ under the LTM, at most one of node $j$'s incoming arcs $(i,j)$ is included in $\mathcal{A}^\omega$ with probability $b_{ij}=n_{ij}/n_j$ (also representing that at least one of the parallel arcs of $(i,j)$
		appears in $\mathcal{G}^{\omega}$).
		We remark that, with more parallel arcs from node $i$ to node $j$, the activation probability $\pi_{ij}$ (under the ICM) or weight $b_{ij}$ (under the LTM) is larger, and thus arc $(i,j)$ will be more likely to be included in graph $\mathcal{G}^\omega$.}
	For the IMP with each combination of the above parameters, 5 instances are randomly generated.
	Therefore, for each \revv{network} in Table \ref{table: instances}, we have 180 and 60 instances for the IMP under the ICM and LTM, respectively.

	In our experiments, we compare the performance of the following three settings:
	\begin{itemize}
		\item $\Default$: solving the IMP based on the Benders reformulation of \eqref{IMP} with the SNA applied \footnote{This setting can be seen as the implementation in \cite{Guney2020}. Unfortunately, we could not access the code of  \cite{Guney2020} online. Therefore, the results reported in this section are based on our implementation. Notice that, however, due to the differences in hardware and randomness in sampling, it cannot be expected that the  results of setting $\Default$ are the same as those in \cite{Guney2020}.};
		\item $\SCNA$: $\Default$ with the SCNA applied;
		\item $\INA$: $\Default$ with the SCNA and INA applied.
	\end{itemize}
	Following \cite{Guney2020}, the memory control parameter $\mathbf{MemLimPerScen}$ is set to $8/|\Omega|$ GB.
	Unless otherwise stated, in the implementation of the INA (i.e.,  Algorithm \ref{algorithm_INA}), parameter $\maxlength$ is set to $8$ and $4$ for the IMP under the ICM and LTM, respectively.
	Finally, to avoid generating too many Benders optimality cuts at fractional points, we follow \cite{Guney2020} to stop the separation procedure if the dual bound improves by less than $0.001$.

	\subsection{Results for the IMP under the ICM}\label{sect: EICM}

	In this subsection, we test the effectiveness of the proposed presolving methods SCNA and INA for the IMP under the ICM.
	%
	Table \ref{table: RC_ICM} reports the reductions by applying the SNA, SCNA, and INA.
	For convenience, we only report the results for the case $|\Omega|=1000$ since the results for the other two cases ($|\Omega|=250, 500$) are similar.
	\revv{For each setting under which the corresponding presolving method(s) are applied, we use $\DeltaZ$ and $\DeltaR$ to represent the reductions in percentage of the numbers of variables $\boldsymbol{z} $ and elements in the  reachability sets of nodes in all live-arc graphs.}
	\revv{In addition, we use $\DeltaZSCNA$ and $\DeltaRSCNA$ to denote the reductions in percentage by applying the SCNA, and  $\DeltaZINA$ and $\DeltaRINA$ to denote the reductions in percentage by applying the INA.}
	For the SCNA, we additionally list the average percentages of nodes and arcs reductions ($\DeltaV=\sum_{\omega\in\Omega}(|\mathcal{V}|-|\mathcal{\bar{V}}^\omega|)/(|\Omega||\mathcal{V}|)$ and $\DeltaA=\sum_{\omega\in\Omega}(|\mathcal{A}^\omega|-|\mathcal{\bar{A}}^\omega|)/(|\Omega||\mathcal{A}^\omega|$))
	to compare the sizes of the compact live-arc graphs $\mathcal{\bar{G}}^{\omega}=(\mathcal{\bar{V}}^{\omega}, \mathcal{\bar{A}}^\omega)$ and the original live-arc graphs $\mathcal{G}^\omega=(\mathcal{V}, \mathcal{A}^\omega)$.
	For the SCNA, the reductions on the numbers of nodes and variables $\boldsymbol{z}$ are equal, and hence we have $\revv{\DeltaZSCNA}=\DeltaV$ under setting $\SCNA$.
	\revv{It is worthwhile remarking that $\DeltaZ$, $\DeltaR$, $\DeltaV$, and $\DeltaA$ can reflect the efficiency of the separation algorithm for the Benders optimality cuts over the one in which no presolving method is applied.
		The larger the $\DeltaZ, \DeltaR, \DeltaV,$ and $\DeltaA$, the more efficient the separation algorithm is.}
	\begin{table}[h]
		\footnotesize
		\centering
		\caption{\revv{The reductions on the sizes of live-arc graphs and numbers of variables $\boldsymbol{z}$ and elements in the reachability sets through applying the SNA, SCNA, and INA (for the IMP  under the ICM).}}
		\label{table: RC_ICM}
		\setlength{\tabcolsep}{3pt}
		\renewcommand{\arraystretch}{1.5}
		\setlength{\tabcolsep}{1.0mm}{
			\begin{tabular}{|ll|ll|lll|ll|}
				\hline
				& &
				\multicolumn{2}{c|}{$\Default$} &
				\multicolumn{3}{c|}{$\SCNA$} &
				\multicolumn{2}{c|}{$\INA$} \\
				\makecell[l]{Network} & \makecell[c]{$p$} & \makecell[l]{$\DeltaZ$} & \makecell[l]{$\DeltaR$}
				&
				$\DeltaZ$~($\DeltaZSCNA/\DeltaV$) & \makecell[l]{$\DeltaR$~($\DeltaRSCNA$)}&
				$\DeltaA$
				&
				\makecell[l]{$\DeltaZ$~($\DeltaZINA$)} & \makecell[l]{$\DeltaR$~($\DeltaRINA$)}
				\\
				\hline
				MSG      & 0.01& 80.9\% & 8.9\%& 83.0\% (2.1\%) & 49.8\% (40.9\%) & 12.9\%& 87.8\% (4.8\%) & 50.9\% (1.1\%)
				\\
				$(\rho=31.5)$& 0.05& 56.4\% & 0.2\%& 82.1\% (25.6\%) & 91.2\% (91.0\%) & 78.8\%& 84.3\% (2.2\%) & 91.3\% ($<$0.1\%)
				\\
				& 0.10& 44.1\% & 0.1\%& 81.3\% (37.2\%) & 95.2\% (95.1\%) & 88.6\%& 82.7\% (1.4\%) & 95.2\% ($<$0.1\%)
				\\
				GNU      & 0.01& 93.1\% & 85.9\%& 93.1\% ($<$0.1\%) & 86.0\% (0.1\%) & 1.0\%& 98.5\% (5.4\%) & 96.0\% (10.0\%)
				\\
				$(\rho=7.4)$& 0.05& 72.3\% & 33.9\%& 73.2\% (0.9\%) & 39.2\% (5.3\%) & 5.1\%& 86.8\% (13.6\%) & 53.0\% (13.8\%)
				\\
				& 0.10& 55.8\% & 0.2\%& 64.1\% (8.3\%) & 59.4\% (59.2\%) & 19.9\%& 74.8\% (10.7\%) & 59.4\% (0.1\%)
				\\
				HEP      & 0.01& 93.3\% & 84.0\%& 93.4\% (0.2\%) & 85.0\% (1.0\%) & 4.6\%& 98.6\% (5.1\%) & 94.6\% (9.6\%)
				\\
				$(\rho=7.7)$& 0.05& 75.8\% & 1.6\%& 80.6\% (4.9\%) & 76.2\% (74.6\%) & 29.7\%& 93.5\% (12.9\%) & 76.8\% (0.6\%)
				\\
				& 0.10& 62.1\% & 0.2\%& 75.5\% (13.4\%) & 88.2\% (88.0\%) & 52.1\%& 91.8\% (16.3\%) & 88.3\% (0.1\%)
				\\
				ENRON    & 0.01& 92.8\% & 21.1\%& 93.0\% (0.1\%) & 35.8\% (14.7\%) & 2.2\%& 95.9\% (2.9\%) & 37.1\% (1.3\%)
				\\
				$(\rho=10.0)$& 0.05& 76.8\% & 0.1\%& 84.4\% (7.6\%) & 73.9\% (73.8\%) & 50.0\%& 90.4\% (6.1\%) & 73.9\% ($<$0.1\%)
				\\
				& 0.10& 64.1\% & $<$0.1\%& 79.7\% (15.5\%) & 81.2\% (81.1\%) & 66.1\%& 86.9\% (7.3\%) & 81.2\% ($<$0.1\%)
				\\
				FACEBOOK & 0.01& 79.1\% & 0.3\%& 81.5\% (2.4\%) & 52.7\% (52.4\%) & 13.9\%& 87.6\% (6.1\%) & 52.7\% ($<$0.1\%)
				\\
				$(\rho=32.4)$& 0.05& 47.8\% & $<$0.1\%& 79.5\% (31.7\%) & 90.8\% (90.8\%) & 77.0\%& 84.1\% (4.6\%) & 90.8\% ($<$0.1\%)
				\\
				& 0.10& 32.5\% & $<$0.1\%& 83.9\% (51.4\%) & 96.2\% (96.2\%) & 90.7\%& 86.8\% (2.9\%) & 96.2\% ($<$0.1\%)
				\\
				DEEZER   & 0.01& 84.4\% & 65.4\%& 84.5\% (0.1\%) & 65.8\% (0.4\%) & 1.0\%& 93.6\% (9.1\%) & 80.0\% (14.2\%)
				\\
				$(\rho=18.3)$& 0.05& 50.6\% & $<$0.1\%& 67.2\% (16.6\%) & 78.8\% (78.8\%) & 39.1\%& 77.1\% (9.9\%) & 78.8\% ($<$0.1\%)
				\\
				& 0.10& 32.4\% & $<$0.1\%& 77.2\% (44.8\%) & 94.2\% (94.2\%) & 77.1\%& 83.0\% (5.8\%) & 94.2\% ($<$0.1\%)
				\\
				TWITTER  & 0.01& 83.8\% & 0.4\%& 86.0\% (2.1\%) & 63.6\% (63.1\%) & 24.4\%& 91.9\% (5.9\%) & 63.6\% (0.1\%)
				\\
				$(\rho=21.7)$& 0.05& 57.9\% & $<$0.1\%& 76.6\% (18.7\%) & 84.3\% (84.3\%) & 67.5\%& 84.1\% (7.5\%) & 84.3\% ($<$0.1\%)
				\\
				& 0.10& 43.1\% & $<$0.1\%& 77.3\% (34.2\%) & 91.3\% (91.3\%) & 83.3\%& 83.7\% (6.4\%) & 91.3\% ($<$0.1\%)
				\\
				EPINIONS & 0.01& 95.9\% & 10.8\%& 96.0\% (0.2\%) & 39.0\% (28.3\%) & 4.1\%& 97.6\% (1.5\%) & 39.4\% (0.3\%)
				\\
				$(\rho=6.4)$& 0.05& 87.7\% & 0.1\%& 91.3\% (3.6\%) & 67.7\% (67.6\%) & 52.5\%& 94.5\% (3.3\%) & 67.7\% ($<$0.1\%)
				\\
				& 0.10& 81.2\% & $<$0.1\%& 87.6\% (6.4\%) & 70.7\% (70.7\%) & 65.9\%& 91.9\% (4.3\%) & 70.8\% ($<$0.1\%)
				\\
				\hline
		\end{tabular}}
	\end{table}

	As it can be seen in Table \ref{table: RC_ICM}, when the network has a small \revv{connectivity} $\rho$ or activation probability $p$, the singleton nodes (nodes without any incoming arc) are more likely to appear in the live-arc graphs.
	As a result, the SNA can eliminate a considerably large numbers of variables and \revv{elements in the reachability sets}.
	In contrast, the SCNA is more effective in eliminating the variables and \revv{elements in the reachability sets} when the network has a relatively large value of $\rho$ or $p$.
	This is reasonable since as $\rho$ or $p$ increases, the original live-arc graphs contain more arcs, and as a result, \revv{more nodes are likely to be strongly connected with other nodes}.
	As for the numbers of nodes and arcs in the compact live-arc graphs, we can observe that they are much smaller than those in the original live-arc graphs, and \revv{in general}, the larger $\rho$ or $p$ is, the more reductions are detected by the SCNA.
	\revv{For the INA,
		%
		we can observe a clear reduction on the number of variables (beyond the SCNA).}
	However, \revv{in most cases}, the reduction on the number of  \revv{elements in the reachability sets} is relatively small, which is due to the fact that most eliminated variables are associated with small reachability sets as we set $\maxlength=8$ in Algorithm \ref{algorithm_INA} (in \revv{Section} \ref{subsec:parameter}, we will perform numerical experiments to confirm that setting $\maxlength=8$ in Algorithm \ref{algorithm_INA} is enough to identify almost all pairs of isomorphic nodes).

	\begin{table}[h]
		\footnotesize
		\centering
		\caption{Comparison of effectiveness of the INA with different numbers of scenarios (for the IMP under the ICM).}\label{table: RC_ICM0.1}
		\setlength{\tabcolsep}{3pt}
		\renewcommand{\arraystretch}{1.5}
		\setlength{\tabcolsep}{1.2mm}{
			\begin{tabular}{|ll|ll|ll|ll|}
				\hline
				\makecell[l]{Network} & \makecell[l]{$|\Omega|$} & \makecell[l]{$\DeltaZINA$} & \makecell[l]{$\DeltaRINA$}
				&
				\makecell[l]{Network} & \makecell[l]{$|\Omega|$} & \makecell[l]{$\DeltaZINA$} & \makecell[l]{$\DeltaRINA$}
				\\
				\hline
				MSG      & 250 & 1.22\% & 0.01\%
				& FACEBOOK & 250 & 2.40\%&$<$0.01\%
				\\
				& 500 & 1.32\% & 0.01\%
				&  & 500 & 2.67\%&$<$0.01\%
				\\
				& 1000& 1.40\% & 0.01\%
				&  & 1000 & 2.88\%&0.01\%
				\\
				GNU      & 250 & 8.32\% & 0.05\%
				& DEEZER & 250 & 4.65\%&$<$0.01\%
				\\
				& 500 & 9.58\% & 0.06\%
				&  & 500 & 5.28\%&$<$0.01\%
				\\
				& 1000& 10.69\% & 0.07\%
				&  & 1000 & 5.81\%&$<$0.01\%
				\\
				HEP      & 250 & 14.73\% & 0.12\%
				& TWITTER & 250 & 5.21\%&$<$0.01\%
				\\
				& 500 & 15.62\% & 0.12\%
				&  & 500 & 5.85\%&$<$0.01\%
				\\
				& 1000& 16.28\% & 0.13\%
				&  & 1000 & 6.38\%&$<$0.01\%
				\\
				ENRON    & 250 & 6.44\% & 0.01\%
				& EPINIONS & 250 & 3.85\%&$<$0.01\%
				\\
				& 500 & 6.92\% & 0.01\%
				&  & 500 & 4.12\%&$<$0.01\%
				\\
				& 1000& 7.27\% & 0.01\%
				&  & 1000 & 4.33\%&0.01\%
				\\
				\hline
		\end{tabular}}
	\end{table}
	
	As discussed in the end of \revv{Section} \ref{sect: INA}, with the increasing number of \revv{scenarios} in the IMP,
	a node in a live-arc graph $\bar{\mathcal{G}}^\omega$ is more likely to be aggregated by other isomorphic nodes (in other live-arc graphs).
	Table \ref{table: RC_ICM0.1} further reports the reductions on the problem size of the SMCLP formulation \eqref{IMP} with different numbers of \revv{scenarios} under setting \INA.
	For convenience, we only report results for the case $p=0.1$ as the results for the other two cases are similar.
	\revv{As observed from Table \ref{table: RC_ICM0.1}, $\DeltaZINA$ and $\DeltaRINA$ tend to slightly increase with the increasing of $|\Omega|$.
		This shows that for the IMP under the ICM, the reductions derived by the INA slightly increase with the increasing number of scenarios.}

	We now evaluate the performance improvement of  the integration of the SCNA and INA with the BD algorithm.
	In Table \ref{table: OP_ICM}, for each of the eight networks, we report the total number of instances that can be solved within the time limit (\sharpS), the average CPU time in seconds (\Time), the average number of  \revv{branch-and-cut nodes}  (\sharpN),
	the average number of added Benders optimality cuts (\sharpC), \revv{the average separation time in seconds (\SepaTime), and the average presolving time in seconds (\PTime)}.
	%
	%
	Notice that the CPU time \Time} includes the presolving time \PTime spent on applying the SCNA/INA \revv{and the separation time \SepaTime}.
\revv{For instances that cannot be solved by any of the three settings, we report the average relative gap (\Gap) in percentage,
defined as $100 \times \frac{{\text{UB}}- {\text{LB}}}{{\text{UB}}}$ where $\text{LB}$ and $\text{UB}$ are the lower bounds and upper bounds returned by CPLEX. ``--'' in a row of Table \ref{table: OP_ICM} denotes that all instances of the corresponding network can be solved by at least one of the three settings.}
Detailed statistics of these results can be found in Tables 1a-8a of \cite{Chen2022}.
As it can be seen in Table \ref{table: OP_ICM}, the performance of setting $\SCNA$ is much better than that of setting $\Default$, especially for instances with large and  \revv{well-connected} networks.
In total, setting $\SCNA$ can solve $1320$ instances with a CPU time of $\revv{289.1}$ seconds,
while setting $\Default$ can only solve $1122$ instances with a CPU time of $\revv{675.1}$ seconds.
\revv{For unsolved instances, the average relative gap reduces from $7.0\%$ to $0.1\%$.}
\revv{The main improvement comes from efficiency of the separation for the Benders optimality cuts since by Proposition \ref{bcnotchanged}, we know that before and after applying the SCNA method, the Benders optimality cuts are identical, and hence the path of the search tree must  be identical.}
The latter is further confirmed by the results of networks MSG and EPINIONS where the numbers of added Benders \revv{optimality} cuts and \revv{branch-and-cut nodes} are identical under settings $\Default$ and $\SCNA$.
Notice that it is reasonable to observe that the numbers of added Benders \revv{optimality} cuts and \revv{branch-and-cut nodes} in other networks are different since some instances cannot be solved within the time limit.

As for setting $\INA$, we observe that \revv{it slightly outperforms setting $\SCNA$}.
In total, setting $\INA$ can solve $11$ more instances than setting $\SCNA$, with the CPU time decreasing from  $\revv{289.1}$  seconds to $\revv{249.5}$ seconds.
This is consistent with the former results in Table \ref{table: RC_ICM} in which only isomorphic nodes with small reachability sets can be detected by the INA and hence its contribution to speed up the solution procedure is not very large.

\revv{From the above results}, we can conclude that for the IMP under the ICM, (i) the SCNA can effectively reduce \revv{the numbers of variables $\boldsymbol{z}$ and elements in the reachability sets and} the sizes of networks and hence is beneficial to solving the IMP especially when the network is large and  \revv{well-connected}; and
(ii) \revv{the INA can further remove a fairly large fraction of variables $\boldsymbol{z}$ from the SMCLP formulation and slightly speed up the solution procedure}.

\begin{sidewaystable}[!htbp]
\centering
\footnotesize
\caption{Performance improvement through  applying the SCNA and INA (for the IMP under the ICM). The best results of {\texttt{\#S}} and {\texttt{T}} among different settings are printed in boldface.
}\label{table: OP_ICM}
\setlength{\tabcolsep}{3pt}
\renewcommand{\arraystretch}{1.5}
\begin{tabular}{|l|llllll|lllllll|lllllll|}
	\hline
	&\multicolumn{6}{c|}{$\Default$}
	&\multicolumn{7}{c|}{$\SCNA$}
	&\multicolumn{7}{c|}{$\INA$}\\
	Network &\makecell[l]{$\sharpS$} &
	\makecell[l]{$\Time$}& $\Gap$  &
	\makecell[l]{$\sharpN$}
	& \makecell[l]{$\sharpC$}&
	$\SepaTime$ &\makecell[l]{$\sharpS$} &
	\makecell[l]{$\Time$}&$\Gap$  & \makecell[l]{$\sharpN$}
	& \makecell[l]{$\sharpC$}&
	$\SepaTime$ &
	\makecell[l]{$\PTime$} &\makecell[l]{$\sharpS$} &
	\makecell[l]{$\Time$}& $\Gap$ & \makecell[l]{$\sharpN$}
	& \makecell[l]{$\sharpC$}&
	$\SepaTime$ &
	\makecell[l]{$\PTime$} \\
	\hline
	
	MSG     & \textbf{180 }& 8.0    & -- & 2   & 1699 & 5.4    & \textbf{180 }& 3.3    & --& 2   & 1699 & 1.2    & 0.3    & \textbf{180 }& \textbf{2.9   } & -- & 0   & 994
	& 1.1    & 0.4    \\
	GNU     & 171 & 98.1   & 0.1 & 18  & 1625 & 51.9   & 172 & 71.8   & $<$0.1 & 18  & 1645 & 27.5   & 1.7    & \textbf{173 }& \textbf{64.0  } & $<$0.1 & 16  & 1599
	& 25.5   & 3.6    \\
	HEP     & 171 & 227.5  & 0.1 & 31  & 4702 & 128.6  & 176 & 138.0  & $<$0.1 & 33  & 4716 & 41.9   & 2.3    & \textbf{177 }& \textbf{95.1  } & $<$0.1 & 33  & 4301
	& 35.1   & 4.0    \\
	ENRON   & 165 & 1138.3 & 0.6 & 14  & 4898 & 731.1  & \textbf{166 }& 484.4  & 0.6 & 14  & 4910 & 230.6  & 5.3    & \textbf{166 }& \textbf{366.8 } & 0.5 & 15  & 5325
	& 209.7  & 7.9    \\
	FACEBOOK& 56  & 5455.0 & 44.5 & 2   & 675  & 4959.4 & 137 & 1855.9 & $<$0.1 & 56  & 2836 & 1431.4 & 9.6    & \textbf{142 }& \textbf{1667.0} & $<$0.1 & 59  & 2524
	& 1343.5 & 14.2   \\
	DEEZER  & 92  & 2866.1 & 11.2 & 1   & 569  & 2367.7 & 144 & 1029.0 & $<$0.1 & 14  & 1624 & 724.2  & 11.0   & \textbf{148 }& \textbf{995.5 } & $<$0.1 & 13  & 1645
	& 711.2  & 21.1   \\
	TWITTER & 107 & 5825.9 & 0.2 & 2   & 550  & 4851.4 & \textbf{165 }& 1456.5 & 0.2 & 2   & 3863 & 973.3  & 17.2   & \textbf{165 }& \textbf{1414.6} & 0.1 & 2   & 3641
	& 962.9  & 29.3   \\
	EPINIONS& \textbf{180 }& 2061.5 & -- & 1   & 2200 & 1812.7 & \textbf{180 }& 851.8  & -- & 1   & 2200 & 664.5  & 20.8   & \textbf{180 }& \textbf{734.1 } & -- & 2   & 3060& 610.1 & 26.1\\
	TOTAL & 1122 & 675.1 & 7.0 & 5 & 1527 & 489.0 & 1320 & 289.1 & 0.1 & 10 & 2663 & 158.5 & 5.6 & \textbf{1331} & \textbf{249.5} & 0.1 & 9 & 2531 & 147.7 & 8.7\\
	\hline
\end{tabular}
\end{sidewaystable}

\subsection{Results for the IMP under the LTM}\label{sect: ELTM}

In this subsection, we present similar computational results for the IMP under the LTM in Tables \ref{table: RC_LTM}-\ref{table: OP_LTM}.
To begin with, we note that for node $j$ in graph $\mathcal{G}$, if it is a singleton node (i.e., it does not have any incoming arc),
then in each live-arc graph $\mathcal{G}^{\omega}$, it is also a singleton node;
otherwise, it has at most one incoming arc in each live-arc graph $\mathcal{G}^{\omega}$ (indeed, in the tested instances, it has exactly one incoming arc as $\sum_{i\,:\,(i,j)\in\A}b_{ij}=1$ holds; see \revv{Section} \ref{settings}).
As a result, for any pair of nodes $i_1$ and $i_2 $ in  live-arc graph $\mathcal{G}^{\omega}$,  there exists at most one directed path from node $i_1$ to node $i_2$.
This implies that in graph $\mathcal{G}^\omega$, the subgraph induced by the nodes in SCC $\mathcal{SC}_u^\omega$ with $|\mathcal{SC}_u^\omega| \geq 2$ must be a single circle
(which is in sharp contrast to the IMP under the ICM in which the subgraph induced by the nodes in SCC $\mathcal{SC}_u^\omega$ with $|\mathcal{SC}_u^\omega| \geq 2$ can be a union of multiple circles).
\revv{This property for the IMP under the LTM, however, implies that only a small portion of nodes are likely to be strongly connected with other nodes, and the reachability sets of the nodes are also likely to be small.}
Consequently, through applying the SCNA, we can only observe a mild reduction on the numbers of nodes and arcs and the numbers of variables and  \revv{elements in the reachability sets} in Table \ref{table: RC_LTM}.
However, due to the small sizes of the reachability sets, more nodes are likely to be isomorphic among different scenarios and hence the INA can detect more reductions, as compared to the IMP under the ICM.
This is shown in Table \ref{table: RC_LTM} in which we observe a fairly large reduction on the number of variables by applying the INA.
The reduction on the number of \revv{elements in the reachability sets} is relatively small which can be explained by the fact that most eliminated variables by the INA are associated with small reachability sets as $\maxlength$ is set to $4$ in our implementation (see \revv{Section} \ref{subsec:parameter} further ahead for the reason of setting $\maxlength=4$).
Notice that in Table \ref{table: RC_LTM}, the reduction detected by the SNA ($\Default$) is marginal since, as it has been mentioned, a node is a singleton node in live-arc graph $\mathcal{G}^{\omega}$ if and only if it is a singleton node in the original graph $\mathcal{G}$.
Therefore, the reductions by applying the SNA for the IMP under the LTM totally depend on the number of singleton nodes in network $\mathcal{G}$, which is very small in most cases.
Indeed, only network EPINIONS contains a relatively large percentage of singleton nodes $(35.9\%)$; see column $\DeltaZ$ \rev{under setting $\Default$} in Table \ref{table: RC_LTM}.
\revv{For the IMP under the LTM, with the increasing number of scenarios, a relatively large increase on the reductions by the INA can be observed, as shown in Table \ref{table: RC_LTM2}.}

\begin{table}[H]
\footnotesize
\centering
\caption{\revv{The reductions on the sizes of live-arc graphs and numbers of variables $\boldsymbol{z}$ and elements in the reachability sets through applying the SNA, SCNA, and INA (for the IMP  under the LTM).}}\label{table: RC_LTM}
\setlength{\tabcolsep}{3pt}
\renewcommand{\arraystretch}{1.5}
\setlength{\tabcolsep}{1.0mm}{
	\begin{tabular}{|l|ll|lll|ll|}
		\hline
		&
		\multicolumn{2}{c|}{$\Default$} &
		\multicolumn{3}{c|}{$\SCNA$} &
		\multicolumn{2}{c|}{$\INA$} \\
		\makecell[l]{Network} & \makecell[l]{$\DeltaZ$} & \makecell[l]{$\DeltaR$}
		&
		$\DeltaZ$~($\DeltaZSCNA/\DeltaV$) & \makecell[l]{$\DeltaR$~($\DeltaRSCNA$)}&
		$\DeltaA$
		&
		\makecell[l]{$\DeltaZ$~($\DeltaZINA$)} & \makecell[l]{$\DeltaR$~($\DeltaRINA$)}
		\\
		\hline
		MSG     & 1.9\% & 0.2\%& 4.3\% (2.4\%) & 20.2\% (20.0\%) & 4.6\%& 14.5\% (10.2\%) & 22.3\% (2.1\%)
		\\
		GNU     & $<$0.1\% & $<$0.1\%& 5.9\% (5.9\%) & 13.4\% (13.4\%) & 11.8\%& 26.2\% (20.3\%) & 18.1\% (4.7\%)
		\\
		HEP     & $<$0.1\% & $<$0.1\%& 22.7\% (22.7\%) & 38.2\% (38.2\%) & 42.6\%& 71.7\% (48.9\%) & 60.3\% (22.1\%)
		\\
		ENRON   & $<$0.1\% & $<$0.1\%& 8.7\% (8.7\%) & 18.4\% (18.4\%) & 16.0\%& 31.3\% (22.7\%) & 23.0\% (4.6\%)
		\\
		FACEBOOK& $<$0.1\% & $<$0.1\%& 2.9\% (2.9\%) & 10.7\% (10.7\%) & 5.5\%& 7.2\% (4.4\%) & 11.0\% (0.3\%)
		\\
		DEEZER  & $<$0.1\% & $<$0.1\%& 4.4\% (4.4\%) & 9.5\% (9.5\%) & 8.5\%& 11.6\% (7.2\%) & 10.2\% (0.7\%)
		\\
		TWITTER & $<$0.1\% & $<$0.1\%& 3.3\% (3.3\%) & 24.3\% (24.3\%) & 5.4\%& 8.4\% (5.2\%) & 24.9\% (0.6\%)
		\\
		EPINIONS& 35.9\% & 6.0\%& 37.7\% (1.8\%) & 13.1\% (7.1\%) & 5.5\%& 57.7\% (20.0\%) & 20.5\% (7.5\%)
		\\
		\hline
\end{tabular}}
\end{table}

\begin{table}[H]
\footnotesize
\centering
\caption{Comparison of effectiveness of the INA with different numbers of scenarios (for the IMP under the LTM).}\label{table: RC_LTM2}
\setlength{\tabcolsep}{3pt}
\renewcommand{\arraystretch}{1.5}
\setlength{\tabcolsep}{1.2mm}{
	\begin{tabular}{|ll|ll|ll|ll|}
		\hline
		\makecell[l]{Network} & \makecell[l]{$|\Omega|$} & \makecell[l]{$\DeltaZINA$} & \makecell[l]{$\DeltaRINA$}
		&
		\makecell[l]{Network} & \makecell[l]{$|\Omega|$} & \makecell[l]{$\DeltaZINA$} & \makecell[l]{$\DeltaRINA$}
		\\
		\hline
		MSG      & 250 & 7.11\% & 1.37\%
		& FACEBOOK & 250 &3.10\% &0.17\%
		\\
		& 500 & 8.61\% & 1.69\%
		&  & 500 & 3.73\%&0.22\%
		\\
		& 1000& 10.19\% & 2.05\%
		&  & 1000 &4.37\% &0.26\%
		\\
		GNU      & 250 & 14.82\% & 3.21\%
		& DEEZER & 250 & 4.82\%&0.45\%
		\\
		& 500 & 17.68\% & 3.95\%
		&  & 500 & 5.99\%&0.58\%
		\\
		& 1000& 20.31\% & 4.68\%
		&  & 1000 & 7.20\%&0.73\%
		\\
		HEP      & 250 & 45.08\% & 19.88\%
		& TWITTER & 250 & 3.52\%&0.40\%
		\\
		& 500 & 47.32\% & 21.16\%
		&  & 500 & 4.35\%&0.50\%
		\\
		& 1000& 48.93\% & 22.11\%
		&  & 1000 & 5.18\%&0.61\%
		\\
		ENRON    & 250 & 17.38\% & 3.38\%
		& EPINIONS & 250 & 18.10\%&6.59\%
		\\
		& 500 & 20.04\% & 4.00\%
		&  & 500 & 19.10\%&7.05\%
		\\
		& 1000& 22.66\% & 4.62\%
		&  & 1000 & 19.98\%&7.46\%
		\\
		\hline
\end{tabular}}
\end{table}

We now present the overall performance improvement of integrating the SCNA and INA into the BD algorithm in Table \ref{table: OP_LTM}.
Detailed statistics of these results can be found in Tables 1b-8b of \cite{Chen2022}.
From Table \ref{table: OP_LTM}, we can see that the performance of settings $\SCNA$ and $\INA$ is slightly better than setting $\Default$.
%
\revv{Indeed, we only observe a minor improvement on the average CPU time (\Time), the number of solved instances (\sharpS), and the average relative gap of the unsolved instances (\Gap) through applying the SCNA and INA.}
This can be explained by the reasons that (i) the reduction on the sizes of networks through applying the SCNA is small (as shown in Table \ref{table: RC_LTM});
(ii) the time spent in implementing the SCNA and INA is relative  large (as shown in column \PTime in Table \ref{table: OP_LTM}); and
(iii) only isomorphic nodes with small reachability sets can be detected by the INA (as shown in column \revv{$\DeltaRINA$ under} setting $\INA$ in Table \ref{table: RC_LTM}).
In addition, we note from Table \ref{table: OP_LTM} that the IMPs under the LTM are \rev{generally} much easier than those under the ICM.
In total, under the LTM, only 16 among 480 instances (3.3\%) cannot be solved by setting $\Default$ within the given time limit
while 318 among 1440 instances (22.1\%) cannot be solved by the same setting under the ICM.

\begin{sidewaystable}[!htbp]
\caption{Performance improvement through  applying the SCNA and INA (for the IMP under the LTM). The best results of {\texttt{\#S}} and {\texttt{T}} among different settings are printed in boldface.}\label{table: OP_LTM}
\centering
\footnotesize
\setlength{\tabcolsep}{3pt}
\renewcommand{\arraystretch}{1.5}
\begin{tabular}{|l|llllll|lllllll|lllllll|}
	\hline
	&\multicolumn{6}{c|}{$\Default$}
	&\multicolumn{7}{c|}{$\SCNA$}
	&\multicolumn{7}{c|}{$\INA$}\\
	Network &\makecell[l]{$\sharpS$} &
	\makecell[l]{$\Time$}& $\Gap$  &
	\makecell[l]{$\sharpN$}
	& \makecell[l]{$\sharpC$}&
	$\SepaTime$ &\makecell[l]{$\sharpS$} &
	\makecell[l]{$\Time$}&$\Gap$  & \makecell[l]{$\sharpN$}
	& \makecell[l]{$\sharpC$}&
	$\SepaTime$ &
	\makecell[l]{$\PTime$} &\makecell[l]{$\sharpS$} &
	\makecell[l]{$\Time$}& $\Gap$ & \makecell[l]{$\sharpN$}
	& \makecell[l]{$\sharpC$}&
	$\SepaTime$ &
	\makecell[l]{$\PTime$} \\
	\hline
	MSG     & \textbf{60  }& 18.5   & -- & 10  & 4240 & 5.9    & \textbf{60  }& \textbf{17.7}   & -- & 10  & 4240 & 5.0    & 0.4    & \textbf{60  }& \textbf{17.7  } & -- & 9   & 4167
	& 5.0    & 0.8    \\
	GNU     & \textbf{60  }& 53.2   & -- & 1   & 1818 & 17.9   & \textbf{60  }& \textbf{49.1  } & -- & 1   & 1818 & 12.4   & 2.4    & \textbf{60  }& 50.5   & -- & 1   & 1853
	& 12.6   & 6.7    \\
	HEP     & \textbf{60  }& 85.6   & -- & 0   & 2664 & 22.7   & \textbf{60  }& 77.0   &-- & 0   & 2664 & 13.3   & 3.0    & \textbf{60  }& \textbf{52.0  } & -- & 0   & 2557
	& 12.7   & 9.9    \\
	ENRON   & \textbf{60  }& 184.1  & -- & 0   & 2048 & 77.7   & \textbf{60  }& 165.0  & -- & 0   & 2048 & 54.6   & 8.0    & \textbf{60  }& \textbf{164.7 } & -- & 0   & 1985
	& 54.8   & 30.5   \\
	FACEBOOK& \textbf{49  }& 1482.5 & 1.3 & 27  & 6634 & 805.0  & \textbf{49  }& 1299.3 & 1.2 & 27  & 6693 & 602.9  & 15.7   & \textbf{49  }& \textbf{1293.8} & 1.1 & 28  & 6791
	& 599.2  & 22.9   \\
	DEEZER  & \textbf{60  }& 289.2  & -- & 0   & 1020 & 146.6  & \textbf{60  }& \textbf{258.1 } & -- & 0   & 1020 & 102.2  & 18.4   & \textbf{60  }& 269.4  & -- & 0   & 1020
	& 102.5  & 31.0   \\
	TWITTER & 55  & 1284.7 & 1.5 & 8   & 4994 & 523.3  & 56  & 1173.5 & 1.5 & 8   & 5036 & 393.6  & 23.2   & \textbf{58  }& \textbf{1155.3} & 0.6 & 8   & 5125
	& 392.4  & 33.4   \\
	EPINIONS& \textbf{60  }& 812.2  & -- & 0   & 2441 & 458.2  & \textbf{60  }& 733.7  & -- & 0   & 2441 & 357.6  & 33.5   & \textbf{60  }& \textbf{631.6 } & -- & 0   & 2411
	& 357.1  & 67.7   \\
	TOTAL & 464 & 228.2 & 1.3 & 2 & 2770 & 94.6 & 465 & 207.2 & 1.2 & 2 & 2776 & 68.8 & 8.1 & \textbf{467} & \textbf{194.9} & 1.0 & 2 & 2758 & 68.5 & 16.5\\
	\hline
\end{tabular}
\end{sidewaystable}

\subsection{Selection of parameter $\maxlength$}
\label{subsec:parameter}

As it has been mentioned in \revv{Section} \ref{subsect: im_pre}, $\maxlength$ is a parameter to achieve a trade-off between the effectiveness and efficiency of implementing the INA: the larger the parameter $\maxlength$, the more isomorphic nodes that might be identified and the higher the computational complexity.
Therefore, in this subsection, we compare the performance of different selections of parameter $\maxlength$.
Tables \ref{table: hashlengthICM} and \ref{table: hashlengthLTM} report the computational results for the IMP under the ICM and LTM, respectively.
For simplicity, we only report the results for the case $|\Omega|=1000$ (for the ICM, we only report the results for the case $p=0.1$).
In the two tables, we use $\Mem_0$ to represent the average memory consumption (in GB) of storing all reachability sets after removing those detected by the SCNA,
and $\Mem$ to denote the average memory consumption (in GB) of only storing the reachability sets with the size restriction.
Instead of storing the reachability sets to obtain the required memories $\Mem_0$ and $\Mem$ (which can be potentially very  large on large-scale networks),
we calculate the total number of elements of the stored reachability sets and convert it to the needed memory size.
In our experiments, we set the size restriction $\maxlength=2,4,8,1000,|\V|$, respectively.
Notice that when $\maxlength=|\V|$, we implement the INA without any size restriction on the reachability sets.
\revv{In Tables \ref{table: hashlengthICM} and \ref{table: hashlengthLTM}, $\TINA$ denotes the average runtime in seconds of implementing the INA.}
\revv{In Table \ref{table: hashlengthICM}, ``--'' (under column $\maxlength=|\V|$) indicates that due to the limited memory, we were not able to construct the whole hashing table to implement the INA.}

\begin{sidewaystable}[!htbp]
\caption{Comparison of effectiveness of the INA with different parameters $\maxlength$ (for the IMP under the ICM).}\label{table: hashlengthICM}
\centering
\small
\setlength{\tabcolsep}{3pt}
\renewcommand{\arraystretch}{1.3}
\setlength{\tabcolsep}{1.6mm}{
	\begin{tabular}{|ll|lll|lll|lll|lll|lll|}
		\hline
		\multicolumn{2}{|c|}{$\maxlength$} & \multicolumn{3}{c|}{$2$}
		& \multicolumn{3}{c|}{$4$}
		& \multicolumn{3}{c|}{$8$}
		& \multicolumn{3}{c|}{$1000$}
		& \multicolumn{3}{c|}{$|\mathcal{V}|$}
		\\
		\hline
		\makecell[l]{Network} &  $\Mem_0$ &
		\makecell[l]{$\Mem$} &
		\makecell[l]{$\TINA$} & \makecell[l]{$\DeltaZINA$} &
		\makecell[l]{$\Mem$} &
		\makecell[l]{$\TINA$} & \makecell[l]{$\DeltaZINA$}&
		\makecell[l]{$\Mem$} &
		\makecell[l]{$\TINA$} & \makecell[l]{$\DeltaZINA$} &
		\makecell[l]{$\Mem$} &
		\makecell[l]{$\TINA$} & \makecell[l]{$\DeltaZINA$} &
		\makecell[l]{$\Mem$} &
		\makecell[l]{$\TINA$} & \makecell[l]{$\DeltaZINA$} \\
		\hline
		MSG      & 1.0  &$<$0.1&$<$0.1  &   1.3\%&$<$0.1&$<$0.1  &   1.4\%&$<$0.1&$<$0.1  &   1.4\%&1.0&1.7  &   1.4\%&1.0&1.7  &   1.4\%\\
		GNU      & 8.9  &$<$0.1&1.6  &  8.1\%&$<$0.1&4.2  &  10.7\%&$<$0.1&6.3  &  10.7\%&0.1&7.3  &  10.7\%&8.9&20.1 &  10.7\%\\
		HEP      & 5.5  &$<$0.1&2.3  &  10.8\%&$<$0.1&4.0  &  15.5\%&$<$0.1&5.1  &  16.3\%&0.1&5.4  &  16.3\%&5.5&12.7 &  16.3\%\\
		ENRON    & 147.9&$<$0.1&3.3  &  5.5\%&$<$0.1&5.3  &  7.2\%&$<$0.1&5.8  &  7.3\%&0.1&6.1  &  7.3\%&147.9&170.6&  7.3\%\\
		FACEBOOK & 753.1&$<$0.1&2.9  &  2.5\%&$<$0.1&4.0  &  2.9\%&$<$0.1&4.1  &  2.9\%&$<$0.1&4.2  &  2.9\%&753.1 &--   &    --\\
		DEEZER   & 911.7&$<$0.1&6.7  &  4.9\%&0.1&11.5 &  5.8\%&0.1&12.1 &  5.8\%&0.1&12.9 &  5.8\%&911.7 &--   &    --\\
		TWITTER  & 1520.0&0.1&11.6 &  5.2\%&0.1&20.8 &  6.3\%&0.1&24.6 &  6.4\%&0.3&27.8 &  6.4\%&1520.0 &--   &    --\\
		EPINIONS & 662.5&0.1&11.2 &  3.8\%&0.1&15.5 &  4.3\%&0.1&16.4 &  4.3\%&0.1&17.2 &  4.3\%&662.5 &--   &    --\\
		\hline
\end{tabular}}
\vspace{2\baselineskip}
\caption{Comparison of effectiveness of the INA with different parameters $\maxlength$ (for the IMP under the LTM).}\label{table: hashlengthLTM}
\setlength{\tabcolsep}{3pt}
\renewcommand{\arraystretch}{1.3}
\setlength{\tabcolsep}{2mm}{
	\begin{tabular}{|ll|lll|lll|lll|lll|lll|}
		\hline
		\multicolumn{2}{|c|}{$\maxlength$} & \multicolumn{3}{c|}{$2$}
		& \multicolumn{3}{c|}{$4$}
		& \multicolumn{3}{c|}{$8$}
		& \multicolumn{3}{c|}{$1000$}
		& \multicolumn{3}{c|}{$|\mathcal{V}|$}
		\\
		\hline
		\makecell[l]{Network} & $\Mem_0$ &
		\makecell[l]{$\Mem$} &
		\makecell[l]{$\TINA$} & \makecell[l]{$\DeltaZINA$} &
		\makecell[l]{$\Mem$} &
		\makecell[l]{$\TINA$} & \makecell[l]{$\DeltaZINA$}&
		\makecell[l]{$\Mem$} &
		\makecell[l]{$\TINA$} & \makecell[l]{$\DeltaZINA$} &
		\makecell[l]{$\Mem$} &
		\makecell[l]{$\TINA$} & \makecell[l]{$\DeltaZINA$} &
		\makecell[l]{$\Mem$} &
		\makecell[l]{$\TINA$} & \makecell[l]{$\DeltaZINA$} \\
		\hline
		MSG      & 0.1  &$<$0.1&0.1  &   2.4\%&$<$0.1&1.0  &  10.2\%&$<$0.1&5.5  &  10.7\%&0.1&12.6 &  10.7\%&0.1&12.6 &  10.7\%\\
		GNU      & 0.4  &$<$0.1&1.1  &   5.5\%&0.1&10.3 &  20.3\%&0.1&40.3 &  21.2\%&0.4&71.8 &  21.2\%&0.4&72.0 &  21.2\%\\
		HEP      & 0.2  &$<$0.1&3.4  &  17.2\%&0.1&15.8 &  48.9\%&0.2&37.6 &  55.4\%&0.2&45.9 &  55.4\%&0.2&45.9 &  55.4\%\\
		ENRON    & 1.3  &$<$0.1&3.5  &   6.1\%&0.2&64.5 &  22.7\%&0.4&162.8&  24.4\%&1.3&359.0&  24.4\%&1.3&359.2&  24.4\%\\
		FACEBOOK & 5.3  &$<$0.1&2.4  &   2.0\%&0.1&17.3 &   4.4\%&0.3&62.3 &   4.5\%&5.3&348.4&   4.5\%&5.3&348.5&   4.5\%\\
		DEEZER   & 3.5  &$<$0.1&4.1  &   3.2\%&0.1&31.6 &   7.2\%&0.4&118.2&   7.4\%&3.5&359.7&   7.4\%&3.5&361.5&   7.4\%\\
		TWITTER  & 5.0  &$<$0.1&2.5  &   1.4\%&0.1&24.3 &   5.2\%&0.6&136.7&   5.9\%&5.0&602.1&   5.9\%&5.0&619.7&   5.9\%\\
		EPINIONS & 3.0  &0.2&23.5 &  12.2\%&0.5&82.0 &  20.0\%&1.0&312.2&  20.5\%&3.0&708.5&  20.5\%&3.0&713.8&  20.5\%\\
		\hline
\end{tabular}}
\\
\end{sidewaystable}

For the ICM, Table \ref{table: hashlengthICM} shows that it requires a prohibitively large memory $\Mem_0$ to store all the reachability sets.
However, when restricting the size of the considered reachability sets to a small value of $\maxlength$,
the memory overhead significantly reduces; see column $\Mem$ in Table \ref{table: hashlengthICM}.
In addition, with the increasing value of $\maxlength$, the improvement on the percentage of the eliminated variables $\DeltaZINA$ becomes smaller and smaller.
Indeed, for networks MSG, GNU, HEP, and ENRON, the proposed algorithm with $\maxlength = 8$ can identify almost all isomorphic nodes as those with $\maxlength = |\V|$.
For the other four networks, setting $\maxlength = 8$ enables to identify almost the same amount of isomorphic nodes as those obtained by setting $\maxlength=1000$.
%

We now discuss the results for the IMP under the LTM in Table \ref{table: hashlengthLTM}.
On one hand, for the IMP under the LTM, since there exists at most one incoming arc for each node in each live-arc graph, the sizes of the reachability sets are likely to be smaller than those for the IMP under the ICM.
This leads to a smaller total memory consumption $\Mem_0$ and a larger memory consumption $\Mem$ when restricting $\maxlength$ to a small value, as compared to those for the IMP under the ICM.
As a result, the computational overhead of implementing the INA is very high, even for a relatively small $\maxlength$ (e.g., $\maxlength=8$).
In addition, in analogy to the IMP under the ICM, with a small value of parameter $\maxlength$, the proposed algorithm can identify almost the same amount of isomorphic nodes as the case $\maxlength=|\V|$.
Therefore, for the IMP under the LTM, we choose $\maxlength=4$ in the implementation of the INA to achieve a trade-off between the performance and the time complexity.

	\section{Extensions}\label{sect: EGP}
	
	In this section, we investigate a generalization of the IMP \eqref{IMP}, which arises from many existing applications including the IMP and its variants \cite{Long2011,nguyen2013budgeted,Wu2019,Zhang2014}, and discuss \revv{the extensions of the proposed SCNA and INA to this generalization}.
	
	The considered generalization is of the form 
	\begin{equation}\label{EGP}
		\begin{aligned}
			\max_{\boldsymbol{y},\,\boldsymbol{z}}\ &f(\boldsymbol{y},\boldsymbol{z}) \\
			\hbox{s.t.}\ &\eqref{connectioncons},\eqref{ybincons},\eqref{zbincons},\\
			& \boldsymbol{y}\in\mathcal{Y},~ \boldsymbol{z}\in\mathcal{Z},
		\end{aligned}
	\end{equation}
	where $f$: $\{0,1\}^{|\V|}\times \{0,1\}^{|\V||\Omega|} \rightarrow \mathbb{R}$, $\mathcal{Y} \subseteq \mathbb{R}^{|\V|}$, and $\mathcal{Z} \subseteq \mathbb{R}^{|\V||\Omega|}$.
	Similar to the IMP \eqref{IMP}, problem \eqref{EGP} is built upon a finite number of live-arc graphs $\G^{\omega}=(\V,\A^{\omega}),~\omega\in\Omega$. 
	However, in contrast to the IMP \eqref{IMP}, problem \eqref{EGP} can flexibly allow any objective function and any constraint in sets $\mathcal{Y}$ and $\mathcal{Z}$.
	Indeed, the IMP \eqref{IMP} can be seen as a special case of problem \eqref{EGP} where 
	$f(\boldsymbol{y},\boldsymbol{z})=\sum_{\omega\in\Omega}p^{\omega}\sum_{i\in\V}z_i^{\omega}$,  $\mathcal{Y}=\left\{\boldsymbol{y}\in\mathbb{R}^{|\V|}\,:\,\sum_{j\in\V}y_j\leq K\right\}$, and $\mathcal{Z}=\mathbb{R}^{|\V||\Omega|}$.
	Due to the flexibility, various variants of the IMP can also be seen as special cases of \eqref{EGP}.
	For instance, by choosing $\mathcal{Y}=\left\{\boldsymbol{y}\in\mathbb{R}^{|\V|}\,:\,\sum_{j\in\V}c_jy_j\leq B\right\}$ and the same $\mathcal{Z}$ and $f(y,z)$ as that of the IMP, problem \eqref{EGP} reduces to the \emph{budgeted influence maximization problem} (BIMP) studied in \cite{nguyen2013budgeted}.
	Here $c_j$, $j \in \V$, is the cost of choosing node $j$ as a seed node and $B$ is the total budget.
	We next present another two special cases of problem \eqref{EGP}.
	\begin{itemize}
		\item The \emph{seed minimization problem} (SMP) \cite{Long2011}. In this problem, $f(\boldsymbol{y},\boldsymbol{z})=-\sum_{j\in\V}y_j$, $\mathcal{Y}=\mathbb{R}^{|\V|}$, and  $\mathcal{Z}=\left\{\boldsymbol{z}\in\mathbb{R}^{|\V||\Omega|}\,:\,\sum_{\omega\in\Omega}p^{\omega}\sum_{i\in\V}z_i^{\omega}\geq D\right\}$ ($D\in\mathbb{R}_{++}$). 
		This problem can be seen as a dual form of the IMP \eqref{IMP}, which minimizes the number of seed nodes with an expected influence coverage $D$ in a \revv{network}.
		\item The \emph{seed minimization problem with probabilistic influence coverage guarantee} (SMPPICG) \cite{Wu2019,Zhang2014}.
		In this problem, $f(\boldsymbol{y},\boldsymbol{z})=-\sum_{j\in\V}y_j$, $\mathcal{Y}=\mathbb{R}^{|\V|}$, $\mathcal{Z}=\text{Proj}_{\bz}(\mathcal{W}) $ where $\mathcal{W}= \left\{(\boldsymbol{z},\boldsymbol{\xi}) \in \mathbb{R}^{|\V||\Omega|}\times \{0,1\}^{|\Omega|}\,:\, \sum\limits_{i\in\V}z_i^{\omega}\geq D\xi^{\omega},~\forall~ \omega\in\Omega,~\sum\limits_{\omega\in\Omega}p^{\omega}\xi^{\omega}\geq 1-\varepsilon\right\}$ ($\varepsilon\in(0,1)$ is the confidence level).
		{Instead of ensuring an expected influence coverage threshold $D$, the problem requires to influence at least $D$ nodes with a probability at least $1-\varepsilon$.}
	\end{itemize}
	
	We next discuss \revv{the extensions of the proposed SCNA and INA  to problem \eqref{EGP}}.
	To proceed, we need the following two realistic assumptions.
	\begin{itemize}
		\item [(i)] $f(\boldsymbol{y},\boldsymbol{z})$ is nondecreasing with respect to variables $\boldsymbol{z}$, i.e.,  if $\bz^1,\bz^2 \in \mathcal{Z}$ and $\bz^1 \leq \bz^2 $, then $f(\boldsymbol{y},\boldsymbol{z}^1) \leq f(\boldsymbol{y},\boldsymbol{z}^2)$.
		\item [(ii)] Set $\mathcal{Z}$ is up-monotone, i.e., if $\bz^1 \in \mathcal{Z}$ and $\bz^1 \leq \bz^2 $, then ${\bz}^2 \in \mathcal{Z}$ as well (such a set is also called a reverse normal set \cite{Tuy2000}).
	\end{itemize}
	The two assumptions imply that when node $i$ is reachable in scenario $\omega$ from some seed nodes
	(i.e., $\sum_{j\in\R(\G^{\omega},i)}y_j\geq 1$), 
	activating node $i$ in scenario $\omega$ (i.e., setting $z_i^{\omega}:=1$)
	provides a better solution for  the decision maker while does not violate his/her requirement.
	It can be easily verified that for the IMP, BIMP, SMP, and SMPPICG, the two assumptions are satisfied.
	\begin{proposition}\label{theorem: EPG}
		Suppose that problem \eqref{EGP}, with assumptions (i) and (ii), has an optimal solution.
		Then there must exist an optimal solution $(\bar{\boldsymbol{y}},\bar{\boldsymbol{z}})$ such that  \eqref{projection_on_z_equation} holds.
	\end{proposition}
	\begin{proof}
		Let $(\bar{\boldsymbol{y}},\bar{\boldsymbol{z}})$ be an optimal solution of problem \eqref{EGP}.
		Suppose that  \eqref{projection_on_z_equation} does not hold for some $\omega_0\in\Omega$ and $i_0\in\V$.
		Then we must have  $\bar{z}^{\omega_0}_{i_0}=0$ and $\sum_{j\in\R(\G^{\omega_0},i_0)}\bar{y}_j\geq 1$.
		Setting  $\bar{z}^{\omega_0}_{i_0}:=1$, we obtain a new point  $(\bar{\boldsymbol{y}}, {\boldsymbol{z}'})$.
		By assumption (ii), ${\boldsymbol{z}'} \in \mathcal{Z}$, and by $\sum_{j\in\R(\G^{\omega_0},i_0)}\bar{y}_j\geq 1$,  constraints \eqref{connectioncons} hold at point $(\boldsymbol{\bar{y}}, {\boldsymbol{z}'})$.
		This implies that  $(\bar{\boldsymbol{y}}, {\boldsymbol{z}'})$ is  a  feasible solution of problem \eqref{EGP}.
		Moreover, by assumption (i), $f(\bar{\boldsymbol{y}}, {\boldsymbol{z}'}) \geq f(\bar{\boldsymbol{y}},\bar{\boldsymbol{z}})$, indicating that $(\bar{\boldsymbol{y}}, {\boldsymbol{z}'})$ must also be an optimal solution of problem \eqref{EGP}.
		Recursively using the above argument, the statement follows.
	\end{proof}
	
	By Proposition \ref{theorem: EPG}, if $\mathcal{R}(\mathcal{G}^\omega,i_1) = \mathcal{R}(\mathcal{G}^\eta,i_2)$, we can set $z_{i_1}^\omega:=z_{i_2}^\eta$ in problem \eqref{EGP} (with the two realistic assumptions (i) and (ii)).
	As a result, the proposed SCNA and INA can also be applied to problem \eqref{EGP} to reduce the problem size and improve the solution efficiency.

	\section{Concluding remarks}
	\label{Sect:conclusions}
	In this paper, we proposed two new presolving methods, called the SCNA and INA, and integrated them into the BD algorithm to solve the IMP.
	The SCNA enables to build an SMCLP formulation for the considered problem based on the (potentially) much more compact live-arc graphs, which are obtained by aggregating strongly connected nodes in the original live-arc graphs.
	The INA further reduces the problem size of the SMCLP formulation by aggregating isomorphic nodes among different live-arc graphs.
	We provided a theoretical analysis on two special cases of the IMP to show the strength of the proposed SCNA and INA in reducing the problem size of the SMCLP formulation.
	%
	Furthermore, with the SCNA and INA, a (potentially) much faster separation procedure for the Benders optimality cuts is developed, which plays a crucial role in speeding up the BD algorithm. 
	We have performed extensive experiments to analyze the performance impact of the proposed SCNA and INA on solving the IMP with real-world \revv{networks}. 
	Computational results show that the proposed SCNA and INA can effectively reduce the problem size\revv{, speed up the separation of Benders optimality cuts, 
		and (hence) improve the \revv{overall} performance of using the BD algorithm to solve the IMP.}
	We also studied a generalization of the IMP and demonstrated that the proposed \revv{SCNA and INA} are applicable to this generalization under some realistic assumptions.

	There still exist some instances where the proposed SCNA and INA cannot effectively reduce the problem size of the SMCLP formulation.
	Indeed, in \revv{Section 4}  of \cite{Chen2022}, we have provided a worst-case example showing that the percentage of the eliminated variables in the SMCLP formulation of the IMP tends to zero with a probability tending to one.
	Consequently, it is interesting to develop more powerful presolving methods for solving the IMP. 
	In addition, it also deserves to investigate
	whether the proposed presolving methods are computationally effective in solving other variants of the IMP \cite{Long2011,nguyen2013budgeted,Wu2019,Zhang2014}.\\[5pt]
	\section*{Acknowledgments} 
	The works of S.-J. Chen and Y.-H. Dai were supported in part by  the National Natural Science Foundation of China (Nos.  12021001, 11991021, 11991020, and 11971372), the National Key  R$\&$D Program of China (Nos. 2021YFA1000300 and 2021YFA1000301), and the Strategic Priority Research Program of Chinese Academy of Sciences (No. XDA27000000).
	The work of W.-K. Chen was supported in part by the National Natural Science Foundation of China (No. 12101048) and Beijing Institute of Technology Research Fund Program for Young Scholars.
	The work of J.-H. Yuan and H.-S. Zhang were supported in part by the National Natural Science Foundation of China (No. 12171052).

	\renewcommand{\refname}{\normalsize References}\small
	\bibliographystyle{abbrvnat}
	\bibliography{socialnetwork}
	
	\appendix
	\titleformat{\section}{\normalfont\large\bfseries}{Appendix \thesection}{1em}{}

	\section{}
	\label{proof_bipartite}
	\section*{Proof of Proposition \ref{theorem:obj}}
	\begin{proof}
		Clearly, if $K \geq |\mathcal{M}|+|\mathcal{N}|$, point $(\boldsymbol{y},\boldsymbol{z})= (\boldsymbol{e},\boldsymbol{e})$ is optimal for formulation \eqref{IMP_obg_presolve} and its LP relaxation, where $\boldsymbol{e}$ is an all-ones vector with appropriate dimension. 
		As a result, the statement follows.
		Therefore, in the following, we consider the case $K < |\mathcal{M}|+|\mathcal{N}|$.
		
		Let $(\bar{\boldsymbol{y}}, \bar{\boldsymbol{z}})$ be an optimal solution of the LP relaxation of formulation \eqref{IMP_obg_presolve}.
		If there exists some $ i_0 \in \mathcal{M}$ and $j_0 \in \mathcal{N}$ such that $\bar{y}_{i_0}< 1$ and $\bar{y}_{j_0} > 0$, then we can construct a new point $(\hat{\boldsymbol{y}}, \hat{\boldsymbol{z}})$ as follows: 
		\begin{itemize}
			\item []
			$\hat{y}_{i_0} := \bar{y}_{i_0} + \varepsilon , ~ 	
			\hat{y}_{j_0} := \bar{y}_{j_0} - \varepsilon, ~\text{and}~ 
			\hat{y}_i := \bar{y}_i~\text{for}~i \in \mathcal{M}\cup \mathcal{N}\backslash\{i_0, j_0\}$;
			\item []
			$\hat{z}_{ij_0} := \max \{\bar{z}_{ij_0} - \varepsilon, 0\}$ for $(i,j_0) \in \mathcal{A}'$ and $\hat{z}_{ij} :=\bar{z}_{ij}$ for $(i,j)\in \mathcal{A}'$ with $j \neq j_0$,
		\end{itemize}
		where $\varepsilon > 0$ is a sufficiently small value.  
		It is easy to see that point $(\hat{\boldsymbol{y}}, \hat{\boldsymbol{z}})$ is also feasible for the LP relaxation of formulation \eqref{IMP_obg_presolve}.
		Moreover, point $(\hat{\boldsymbol{y}}, \hat{\boldsymbol{z}})$ must be optimal since
		\begin{equation*}
			\begin{aligned}
				&  \sum_{i\in\mathcal{M}\cup\mathcal{N}}s_i\hat{y}_i+\sum_{(i,j)\in\mathcal{A}'}c_{ij}\hat{z}_{ij} - 
				\left (\sum_{i\in\mathcal{M}\cup\mathcal{N}}s_i\bar{y}_i+\sum_{(i,j)\in\mathcal{A}'}c_{ij}\bar{z}_{ij}\right )\\
				=&   s_{i_0}(\bar{y}_{i_0}+\varepsilon) +s_{j_0} (\bar{y}_{j_0}-\varepsilon) +  \sum_{i\,:\,(i,j_0)\in \mathcal{A}'}c_{ij_0}\max\{ \bar{z}_{ij_0}-\varepsilon , 0\} \\ 
				&\qquad \qquad \qquad\qquad\qquad~~~~ -s_{i_0}\bar{y}_{i_0}  - s_{j_0} \bar{y}_{j_0}- \sum_{i\,:\,(i,j_0)\in \mathcal{A}'} c_{ij_0} \bar{z}_{ij_0}\\
				=&  s_{i_0}\varepsilon  - s_{j_0} \varepsilon + \sum_{i\,:\,(i,j_0)\in \mathcal{A}'}c_{ij_0}\max\{ -\varepsilon , - \bar{z}_{ij_0}\}  \\
				\geq&  s_{i_0} \varepsilon -s_{j_0}  \varepsilon  - \sum_{i\,:\,(i,j_0)\in \mathcal{A}'} c_{ij_0}\varepsilon=  \varepsilon \left (s_{i_0} -  s_{j_0} - \sum_{i\,:\,(i,j_0)\in \mathcal{A}'} c_{ij_0}\right) =0,
			\end{aligned}
		\end{equation*}
		where the last equality follows from Remark \ref{defsc}.
		Recursively applying the above argument, we will obtain an optimal solution $(\tilde{\boldsymbol{y}},\tilde{\boldsymbol{z}})$ of the LP relaxation of formulation \eqref{IMP_obg_presolve} fulfilling:
		\begin{quote} 
			($\star$) ~~if $\tilde{y}_{j_0}> 0$ for some $ j_0 \in \mathcal{N}$, then $\tilde{y}_i = 1$ for all $i \in \mathcal{M}$ must hold.
		\end{quote}
		Furthermore, we can, without loss of generality, assume the followings on point $(\tilde{\boldsymbol{y}},\tilde{\boldsymbol{z}})$.
		\begin{itemize}
			\item [1)] $\sum_{i \in \mathcal{M}\cup \mathcal{N}} \tilde{y}_i = K $. Otherwise,  we can increase \revv{some} $\tilde{y}_i$, with $\tilde{y}_i < 1$, without decreasing the objective value (as $\sum_{i \in \mathcal{M}\cup \mathcal{N}} \tilde{y}_i  < K < |\mathcal{M}|+|\mathcal{N}|$);
			\item [2)] $\tilde{z}_{ij} =  \min\{ \tilde{y}_i + \tilde{y}_j, 1 \} $ for all $(i,j)\in \mathcal{A}'$. Otherwise, we can increase $\tilde{z}_{ij}$, with $\tilde{z}_{ij} < \min\{ \tilde{y}_i + \tilde{y}_j, 1 \} $, without decreasing the objective value.  
		\end{itemize}
		Together with ($\star$) and 1),  point $(\tilde{\boldsymbol{y}}, \tilde{\boldsymbol{z}})$ must satisfy the followings
		\begin{itemize}
			\item [(i)] if $K \geq  |\mathcal{M}|$, then $\tilde{y}_i = 1$ for all $i \in \mathcal{M}$; and 
			\item [(ii)] if $K < |\mathcal{M}|$, then $\tilde{y}_j = 0$ for all $j \in \mathcal{N}$.
		\end{itemize}
		We next prove the statement in the proposition  by treating cases (i) and (ii) separately.
		
		(i) In this case, $\tilde{y}_i=1$ for all $i\in\mathcal{M}$ and by 2), $ \tilde{z}_{ij}=1 $ for all $(i,j)\in \mathcal{A}'$. Setting $y_i := 1$  for all $i\in\mathcal{M}$ and $ {z}_{ij}:=1 $ for all $(i,j)\in \mathcal{A}'$, the LP relaxation of formulation \eqref{IMP_obg_presolve} reduces to
		\begin{equation}\label{IMP_obj_case1}
			\begin{aligned}
				\max_{\boldsymbol{y}} \left \{\sum_{j\in\mathcal{N}}s_jy_j+|\mathcal{M}|+\sum_{(i,j)\in\mathcal{A}'}c_{ij}\,:\, \sum_{j\in \mathcal{N}}y_j = K-|\mathcal{M}|, ~ y_j \in[0,1] ,~ \forall ~j\in \mathcal{N}\right\}.
			\end{aligned}
		\end{equation}
		Then point $\{ \tilde{y}_j \}_{j\in \mathcal{N}}$ must \revv{be an optimal solution of} formulation \eqref{IMP_obj_case1}.
		On the other hand, suppose that $s_{j_1} \geq \cdots \geq s_{j_{|\mathcal{N}|}} $ where $\{ j_1, \ldots, j_{|\mathcal{N}|} \} = \mathcal{N}$.
		It is easy to show that point $\{ y'_j \}_{j \in \mathcal{N}}$ is also \revv{an optimal solution of} formulation \eqref{IMP_obj_case1} where $y'_{j_\tau}=1$ for $\tau = 1, \ldots, K-|\mathcal{M}| $ and $y'_{j_\tau} = 0$ otherwise.
		We next extend point $\{ y'_j \}_{j \in \mathcal{N}}$ to a higher dimensional point $(\boldsymbol{y}',\boldsymbol{z}')\in \{0,1\}^{|\mathcal{M}|+|\mathcal{N}|}\times \{0,1\}^{|\mathcal{A}'|}$ by additionally setting $y'_i:=1$ for all $i \in \mathcal{M}$ and $z'_{ij}:=1$ for all $(i,j) \in \mathcal{A}'$.
		The 0-1 point $(\boldsymbol{y}',\boldsymbol{z}') $ must be optimal to formulation \eqref{IMP_obg_presolve} and its LP relaxation.
		This implies that the LP relaxation of formulation \eqref{IMP_obg_presolve} is tight and formulation \eqref{IMP_obg_presolve} is strongly \revv{polynomial time} solvable for this case.
		
		(ii) In this case, $\tilde{y}_j=0$ for all $j\in\mathcal{N}$ and by 2), $ \tilde{z}_{ij}=\tilde{y}_i $ for all $(i,j)\in \mathcal{A}'$. 
		Setting $y_i := 0$  for all $i\in\mathcal{N}$ and $ z_{ij}:=y_i $ for all $(i,j)\in \mathcal{A}'$, the LP relaxation of formulation \eqref{IMP_obg_presolve} reduces to
		\begin{equation}\label{IMP_obj_case2}
			\begin{aligned}
				\max_{\boldsymbol{y}} \left \{\sum_{i\in\mathcal{M}} \left (1 + \sum_{j\,:\,(i,j)\in \mathcal{A}'} c_{ij}\right )y_i\,:\,\sum_{i\in \mathcal{M}}y_i = K, ~ y_i \in[0,1] ,~ \forall ~i\in \mathcal{M}\right\}.
			\end{aligned}
		\end{equation}
		Then point $\{ \tilde{y}_i\}_{i\in \mathcal{M}}$ must be \revv{an optimal solution of} formulation \eqref{IMP_obj_case2}.
		Using a similar argument in case (i), we can also prove the statement in this case.
	\end{proof}
	
	\section{}
	\label{proof_scna_BD}
	\section*{Proof of Proposition \ref{bcnotchanged}}
	
	\begin{proof}
		\revv{We first prove case (i).
			Note that for all nodes in a given SCC $\mathcal{SC}_u^{\omega}$ of $\G^{\omega}$ ($u \in \bar{\mathcal{V}}^\omega$), their reachability sets are identical and equal to $\mathcal{R}(\mathcal{G}^\omega,\mathcal{SC}_u^{\omega})$.
			As a result,  $(\hat{\alpha}_{i}^{\omega},
			\hat{\beta}_{i}^{\omega})$ for all $i \in \mathcal{SC}_u^{\omega}$ must also be identical, denoted by $(\hat{\alpha}_{u}^{\omega},\hat{\beta}_{u}^{\omega})$. 
			Hence, inequality \eqref{bc1} can be rewritten as
			\begin{equation}\label{bc3}
				\begin{aligned}
					\varphi^{\omega} \leq& \sum_{u\in\bar{\mathcal{V}}^{\omega}} \left(|\mathcal{SC}_u^{\omega}|\hat{\alpha}_{u}^{\omega}\sum_{j\in\mathcal{R}(\mathcal{G}^\omega,\mathcal{SC}_u^{\omega})}y_j+|\mathcal{SC}_u^{\omega}|\hat{\beta}_{u}^{\omega}\right)\\ 
					=&\sum_{u\in\bar{\mathcal{V}}^{\omega}}\left(|\mathcal{SC}_u^{\omega}|\hat{\alpha}_{u}^{\omega}\sum_{v\in\mathcal{R}(\bar{\mathcal{G}}^\omega,u)} y(\mathcal{SC}_v^\omega)+|\mathcal{SC}_u^{\omega}|\hat{\beta}_{u}^{\omega}\right),
				\end{aligned}
			\end{equation}
			where the equality follows from \eqref{relation_reachset}.
			By \eqref{relation_reachset} and the definitions of  $(\hat{\boldsymbol{\alpha}}^{\omega},\hat{\boldsymbol{\beta}}^{\omega})$ and $(\bar{\boldsymbol{\alpha}}^{\omega},\bar{\boldsymbol{\beta}}^{\omega})$ in \eqref{ab1} and \eqref{ab2}, we must have  $(\bar{\alpha}_u^{\omega},\bar{\beta}_u^{\omega})=(|\mathcal{SC}_u^{\omega}|\hat{\alpha}_{u}^{\omega},|\mathcal{SC}_u^{\omega}|\hat{\beta}_{u}^{\omega})$ for all $u \in \bar{\mathcal{V}}^\omega$.
			Thus, the two Benders optimality cuts \eqref{bc4} and \eqref{bc3} are equivalent.}

		We next prove case (ii). 
		Indeed, for two different scenarios $\omega$ and $\eta$, the corresponding Benders optimality cuts based on formulation \eqref{strongIMP} are given by
		\begin{equation}\label{INAChangebefore1}
			\varphi^{\omega}\leq\sum_{u\in\bar{\mathcal{V}}^{\omega}}\left(\bar{\alpha}_u^{\omega}\sum_{v\in\mathcal{R}(\bar{\mathcal{G}}^{\omega},u)}y(\mathcal{SC}_v^{\omega})+\bar{\beta}_u^{\omega}\right)~\text{and}
		\end{equation}
		\begin{equation}\label{INAChangebefore2}
			\varphi^{\eta}\leq\sum_{u\in\bar{\mathcal{V}}^{\eta}}\left(\bar{\alpha}_u^{\eta}\sum_{v\in\mathcal{R}(\bar{\mathcal{G}}^{\eta},u)}y(\mathcal{SC}_v^{\eta})+\bar{\beta}_u^{\eta}\right).
		\end{equation}
		Suppose that node $u_1$ in graph $\mathcal{\bar{G}}^\omega$ is isomorphic to node $u_2$ in graph $\mathcal{\bar{G}}^{\eta}$. 
		By applying the INA, we can remove, \rev{for example}, variable $z_{u_1}^{\omega}$ and the corresponding reachability constraint in \eqref{strongconnectioncons} from formulation \eqref{strongIMP}.
		The new objective coefficient $f_{u_2}^{\eta}$ is set to the sum of the old objective coefficients of variables $z_{u_1}^{\omega}$ and $z_{u_2}^{\eta}$.
		%
		As a result, the term $\bar{\alpha}_{u_1}^{\omega}\sum_{v\in\mathcal{R}(\bar{\mathcal{G}}^{\omega},u_1)}y(\mathcal{SC}_v^{\omega})+\bar{\beta}_{u_1}^{\omega}$ in Benders optimality cut \eqref{INAChangebefore1}
		will be incorporated into Benders optimality cut \eqref{INAChangebefore2}, leading to two new Benders optimality cuts
		\begin{equation}\label{INAChangeafter1}
			\varphi^{\omega}\leq\sum_{u\in\bar{\mathcal{V}}^{\omega}\setminus \{u_1\}}\left(\bar{\alpha}_u^{\omega}\sum_{v\in\mathcal{R}(\bar{\mathcal{G}}^{\omega},u)}y(\mathcal{SC}_v^{\omega})+\bar{\beta}_u^{\omega}\right)~\text{and}
		\end{equation}
		\begin{equation}\label{INAChangeafter2}
			\varphi^{\eta}\leq\sum_{u\in\bar{\mathcal{V}}^{\eta}}\left(\bar{\alpha}_u^{\eta}\sum_{v\in\mathcal{R}(\bar{\mathcal{G}}^{\eta},u)}y(\mathcal{SC}_v^{\eta})+\bar{\beta}_u^{\eta}\right)+\left(\bar{\alpha}_{u_1}^{\omega}\sum_{v\in\mathcal{R}(\bar{\mathcal{G}}^{\omega},u_1)}y(\mathcal{SC}_v^{\omega})+\bar{\beta}_{u_1}^{\omega}\right).
		\end{equation}
		Obviously, the new Benders optimality cuts \eqref{INAChangeafter1} and \eqref{INAChangeafter2} may be different from the old ones \eqref{INAChangebefore1} and \eqref{INAChangebefore2}, respectively.
	\end{proof}
	
\end{document}